\documentclass[reqno]{amsart}
\usepackage{}
\usepackage{stmaryrd}
\usepackage{mathrsfs}
\usepackage{cases}
\usepackage{amsfonts}
\usepackage{amssymb}
\usepackage{amsmath}
\usepackage{tikz}
\newtheorem{theorem}{Theorem}[section]
\newtheorem{lemma}[theorem]{Lemma}
\newtheorem{proposition}[theorem]{Proposition}
\newtheorem{corollary}[theorem]{Corollary}
\theoremstyle{definition}
\newtheorem{definition}[theorem]{Definition}

\newtheorem{remark}[theorem]{Remark}
\numberwithin{equation}{section}

\newcommand{\blankbox}[2]

\begin{document}
\title[Inclusion relations between Wiener amalgam and classical spaces]
{Characterization of inclusion relations between Wiener amalgam and some classical spaces$^*$ }
\author{WEICHAO GUO}
\address{School of Mathematics and Information Sciences, Guangzhou University, Guangzhou, 510006, P.R.China}
\email{weichaoguomath@gmail.com}
\author{HUOXIONG WU}
\address{School of Mathematical Sciences, Xiamen University,
Xiamen, 361005, P.R. China} \email{huoxwu@xmu.edu.cn}
\author{QIXIANG YANG}
\address{School of Mathematics and Statics, Wuhan University,
Wuhan, 430072, China}
\email{qxyang@whu.edu.cn}
\author{GUOPING ZHAO}
\address{School of Applied Mathematics, Xiamen University of Technology,
Xiamen, 361024, China} \email{guopingzhaomath@gmail.com}

\thanks{$^*$ Partly supported by the NNSF of China (Grant Nos. 11371295, 11471041) and the NSF of Fujian Province of China (No. 2015J01025).}

\begin{abstract}
In this paper, we establish the sharp conditions for the inclusion relations between Besov spaces $B_{p,q}$ and Wiener amalgam spaces $W_{p,q}^s$.
We also obtain the optimal inclusion relations between local hardy spaces $h^p$ and Wiener amalgam spaces $W_{p,q}^s$,
which completely improve and extend the main results obtained by Cunanana, Kobayashib and Sugimotoa in [J. Funct. Anal. 268 (2015), 239-254].
In addition, we establish some mild characterizations of inclusion relations between Triebel-Lizorkin and Wiener amalgam spaces, which relates some modern inequalities to classical inequalities.
\end{abstract}

\maketitle

\section{INTRODUCTION}
Wiener amalgam spaces are a class of function spaces which amalgamates local properties with global properties.
Specific cases of these spaces was first introduced by Norbert Wiener in \cite{Wiener 1, Wiener 2, Wiener 3}.
In 1980s, H. G. Feichtinger developed a far-reaching generalization of amalgam spaces,
which allows a wide range of Banach spaces to serve as local or global components.
Feichtinger used $W(B,C)$ to denote the Wiener-type spaces, which $B$ and $C$ are served as the local and global component respectively.
Here, we consider a limited case, namely, the Wiener amalgam spaces $W_{p,q}^s$,
which can be re-expressed as $W(\mathscr{F}^{-1}L_q^s,L_p)$ by the notation of Feichtinger.
In addition, modulation space $M_{p,q}^s$ can be re-expressed by $\mathscr{F}^{-1}W(\mathscr{F}L_p,L_q^s)$.

By the frequency-uniform localization techniques, the modulation space can be viewed as a Besov-type space associated with a uniform decomposition (see Triebel \cite{Triebel_modulation space}, Wang-Huang \cite{WH_JDE_2007} ).
Like the modulation space, Wiener amalgam space $W_{p,q}^s$ can be viewed as a Triebel-Lizorkin-type space corresponding to a uniform decomposition,
we refer the readers to \cite{Triebel_modulation space} for this topic.
Due to the difference between uniform and dyadic decompositions, Wiener amalgam space $W_{p,q}$ has many properties different from the corresponding dyadic space.
For instance, the uniform multiplier $e^{i|D|^{\beta}}(0<\beta<1)$ is unbounded in the classical Lebesgue space $L_p$
but bounded on $W_{p,q}^s$ for $1\leqslant p,q\leqslant \infty$, $s\in \mathbb{R}$, one can see \cite{AKKL_JFA_2007, Cunanan_uimodular_wiener} for more details.
Thus, it is interesting to compare the Wiener amalgam spaces with the classical function spaces constructed by dyadic decompositions.
One basic but important problem is how to characterize the inclusion relations between theses two kinds of spaces.
A lot of attentions have been paid to this topic, for example, one can see \cite{Sugimoto_Tomita, Toft_Continunity, WH_JDE_2007} for the inclusion relations between modulation and Besov spaces, \cite{Guo_Zhao_interpolation, Wang_Han, Toft_embedding} for the more general inclusion relations under the frame of $\alpha$-modulation space, \cite{Lcal hardy and modulation_studia} for the inclusion relations between local Hardy spaces $h_p$ and modulation spaces, and \cite{sobolev and modulation} for the inclusion relation between Sobolev and modulation spaces, etc. In particular, Cunanana, Kobayashib and Sugimotoa \cite{Cunanan_embedding_Lp_wiener}
recently give some sufficient and necessary conditions for the inclusion relations between $W_{p,q}^s$ and $L_p$ in the range of $1\leqslant p,q \leqslant \infty$, they also give a corollary for the inclusion relations between $W_{p,q}^s$ and $B_{p,q}$ (the Besov space $B_{p,q}^0$, see (2.16) for the definition). We will recall the main results in \cite{Cunanan_embedding_Lp_wiener} after introducing following notations

Denote $\alpha(p,q)=0\vee n(1-1/p-1/q)\vee n(1/2-1/q)$, $\beta(p,q)=0\wedge  n(1-1/p-1/q)\wedge n(1/2-1/q)$, that is,
\begin{equation*}
  \alpha(p,q)=
  \begin{cases}
  0,  &\text{if } (1/p,1/q)\in A_1:\, 1/q\geqslant (1-1/p)\vee 1/2,\\
  n(1-1/p-1/q), &\text{if } (1/p,1/q)\in A_2:\, 1/p\leqslant (1-1/q)\wedge 1/2,\\
  n(1/2-1/q), &\text{if } (1/p,1/q)\in A_3:\, 1/q\leqslant 1/2\leqslant 1/p;
  \end{cases}
\end{equation*}
\begin{equation*}
  \beta(p,q)=
  \begin{cases}
  0,  &\text{if } (1/p,1/q)\in B_1:\, 1/q\leqslant (1-1/p)\wedge 1/2,\\
  n(1-1/p-1/q), &\text{if } (1/p,1/q)\in B_2:\, 1/p\geqslant (1-1/q)\vee 1/2, \\
  n(1/2-1/q), &\text{if } (1/p,1/q)\in B_3:\, 1/p\leqslant 1/2\leqslant 1/q.
  \end{cases}
\end{equation*}
One can see following figures for visualization.

\begin{tikzpicture}[scale=3]
\coordinate (Origin) at (0,0);
\coordinate(horizontal_axis) at (1.2,0);
\coordinate(vertical_axis) at (0,1.2);
\coordinate (m01) at (0,1);
\coordinate (m0505) at (0.5,0.5);
\coordinate (m11) at (1.2,1.2);
\coordinate (m050) at (0.5,0);
\coordinate (m105) at (1.2,0.5);
\draw[thick,->] (0,0) -- (horizontal_axis) node[right]{$\frac{1}{p}$};
\draw[thick,->] (0,0) -- (vertical_axis) node[above]{$\frac{1}{q}$};
\draw(0,0)node[left]{$0$};
\draw [very thick] (m0505)-- (m01) node [left] {$1$};
\draw [very thick] (m0505)-- (m050) node [below] {$\frac{1}{2}$};
\draw [very thick] (m0505)-- (m105);
\draw(0.6,0.8)node{$A_1$};
\draw(0.25,0.4)node{$A_2$};
\draw(0.7,0.3)node{$A_3$};
\node at (0.6,-0.25){\small Figure 1};
\coordinate (2Origin) at (2,0);
\coordinate(2horizontal_axis) at (3.2,0);
\coordinate(2vertical_axis) at (2,1.2);
\coordinate (2m10) at (3,0);
\coordinate (2m0505) at (2.5,0.5);
\coordinate (2m11) at (3.2,1.2);
\coordinate (2m051) at (2.5,1.2);
\coordinate (2m005) at (2,0.5);
\draw[thick,->] (2Origin) -- (2horizontal_axis) node[right]{$\frac{1}{p}$};
\draw[thick,->] (2Origin) -- (2vertical_axis) node[above]{$\frac{1}{q}$};
\draw(2,0)node[left]{$0$};
\draw [very thick] (2m0505)-- (2m10) node [below] {$1$};
\draw [very thick] (2m0505)-- (2m051);
\draw [very thick] (2m0505)-- (2m005);
\draw(2.8,0.6)node{$B_2$};
\draw(2.35,0.3)node{$B_1$};
\draw(2.25,0.7)node{$B_3$};
\node at (2.6,-0.25){\small Figure 2};

\end{tikzpicture}

We recall the main results in \cite{Cunanan_embedding_Lp_wiener} as follows.
\\\\
\textbf{Theorem A}(cf. \cite{Cunanan_embedding_Lp_wiener})
Let $1\leqslant p,q\leqslant \infty$, $s\in \mathbb{R}$. Then
\begin{enumerate}
 \item
  $W_{p,q}^{s} \subset B_{p,q}$
  if
  $s> \alpha(p,q)$.
  Conversely, if
  $W_{p,q}^{s} \subset B_{p,q}$,
  then
  $s\geqslant \alpha(p,q)$.
 \item
  $B_{p,q} \subset W_{p,q}^{s}$
  if
  $s< \beta(p,q)$.
  Conversely, if
  $B_{p,q} \subset W_{p,q}^{s}$,
  then
  $s\leqslant \beta(p,q)$.
\end{enumerate}
~\\
\textbf{Theorem B}(cf. \cite{Cunanan_embedding_Lp_wiener})
Let $1\leqslant p,q\leqslant \infty$, $s\in \mathbb{R}$. Then,
\begin{enumerate}
 \item
 The following statements about $W_{p,q}^{s} \subset L_p$ are true.
  \begin{enumerate}
  \item For $p\neq \infty , 1/q\geqslant 1/p\wedge 1/2$, $W_{p,q}^{s} \subset L_p$ if and only if $s\geqslant \alpha(p,q)$.
  \item For $p\neq \infty$, $1/q<1/p\wedge  1/2$,  $W_{p,q}^{s} \subset L_p$ if $s> \alpha(p,q)$, and $W_{p,q}^{s} \subset L_p$ implies $s\geqslant \alpha(p,q)$.
  \item For $p= \infty$, $W_{p,q}^{s} \subset L_p$ if and only if $s\geqslant \alpha(p,q)$ with strict inequality when $q\neq1$.
  \end{enumerate}
 \item
 The following statements about $L_p \subset W_{p,q}^s$ are true.
  \begin{enumerate}
  \item For $p\neq 1, 1/q\leqslant 1/p\vee 1/2,$, $L_p \subset W_{p,q}^s$ if and only if $s\leqslant \beta(p,q)$.
  \item For $p\neq 1$, $1/2\vee 1/p<1/q$, $L_p \subset W_{p,q}^s$ if $s< \beta(p,q)$, and $L_p \subset W_{p,q}^s$ implies $s\leqslant \beta(p,q)$.
  \item For $p= 1$, $L_p \subset W_{p,q}^s$ if and only if $s\leqslant \beta(p,q)$ with strict inequality when $q\neq \infty$.
  \end{enumerate}
\end{enumerate}

Clearly, there exists a gap in Theorem A, or (b) of Theorem B for the endpoint: $s=\alpha(p,q)$ or $s=\beta(p,q)$, namely, it remains to be open whether $s=\alpha(p,q)$ or $s=\beta(p,q)$ is sufficient for the corresponding inclusion relations. Moreover, what about the inclusion relations, in Theorems A and B, for the case: $0<p,\,q<1$? The main purpose of this paper is to address the questions above. Our results can be formulated as follows.


\begin{theorem}\label{theorem, embedding between Wiener and Besov}
Let $0<p,q\leqslant \infty$, $s\in \mathbb{R}$. Then
\begin{enumerate}
 \item
  $W_{p,q}^{s} \subset B_{p,q}$
  if and only if
  $s\geqslant \alpha(p,q)$ with strict inequality when $1/p< 1/q$.
 \item
  $B_{p,q} \subset W_{p,q}^{s}$
  if and only if
  $s\leqslant \beta(p,q)$ with strict inequality when $1/p> 1/q$.
\end{enumerate}
\end{theorem}

\begin{theorem}\label{theorem, embedding between Wiener and hp}
Let $0<p<\infty$, $0<q\leqslant \infty$, $s\in \mathbb{R}$.  Then
\begin{enumerate}
 \item
  $W_{p,q}^{s} \subset h_p$
  if and only if
  $s\geqslant \alpha(p,q)$
  with strict inequality when
  $1/q<1/p\wedge 1/2$.
  \item
  $h_p \subset W_{p,q}^{s}$
  if and only if
  $s\leqslant \beta(p,q)$
  with strict inequality when
  $1/q> 1/p\vee 1/2$.
\end{enumerate}
\end{theorem}

\begin{theorem}\label{theorem, embedding between Wiener and L1 or L infty}
Let $0<q\leqslant \infty$, $s\in \mathbb{R}$. Then,
\begin{enumerate}
 \item
  $W_{1,q}^{s} \subset L_1$
  if and only if
  $s\geqslant \alpha(1,q)$
  with strict inequality when
  $1/q<1/2$.
  \item
  $W_{\infty,q}^{s} \subset L_{\infty}$
  if and only if
  $s\geqslant \alpha(\infty,q)$
  with strict inequality when
  $1/q<1$.
  \item
  $L_1 \subset W_{1,q}^{s}$
  if and only if
  $s\leqslant \beta(1,q)$
  with strict inequality when
  $q\neq \infty$.
  \item
  $L_{\infty} \subset W_{\infty,q}^{s}$
  if and only if
  $s\leqslant \beta(\infty,q)$
  with strict inequality when
  $1/q>1/2$.
\end{enumerate}
\end{theorem}

\begin{remark}\label{rek1}We would like to compare our results with the corresponding results in \cite{Cunanan_embedding_Lp_wiener}.
Firstly, Theorem \ref{theorem, embedding between Wiener and Besov} is a complete improvement and extension of Theorem A in the following two aspects:
(i) Relax the range of $(p,q)$ from $[1,\infty]^2$ to the full range $(0,\infty]^2$;
(ii) Make it clear whether the inclusion between $W_{p,q}^s$ and $B_{p,q}$ is correct or not in the critical case $s=\alpha(p,q)$ or $\beta(p,q)$.
Thus, Theorem \ref{theorem, embedding between Wiener and Besov} is sharper than Theorem A, even in the range $1\leqslant p, q\leqslant \infty$,
and the results in Theorem 1.1 are new for $p\in (0,1)$ or $q\in (0,1)$.

Secondly, compared Theorem \ref{theorem, embedding between Wiener and hp} with Theorem B, it is obvious that the conclusions of Theorem 1.2 are new
for $p\in (0,1]$ or $q\in (0,1]$. For $1\leqslant p \leqslant \infty$, $1\leqslant q\leqslant \infty$, we make some comments as follows.
Recall that $h_p$ is equivalent to $L_p$ for $p\in (1,\infty)$,  Theorems \ref{theorem, embedding between Wiener and hp} and \ref{theorem, embedding between Wiener and L1 or L infty}
actually solve all the problems remaining open in Theorem B.
Also, the inclusion relation $W_{1,q}\subset L_1$ in $1/q\geqslant 1/2$
in Theorem B is improved to $W_{1,q}\subset h_1$ in Theorem \ref{theorem, embedding between Wiener and hp}. In addition, $L_1\subset W_{1,q}$ was proved to be not right if $q\neq \infty$,
however, we prove that a substitute inclusion relation $h_1\subset W_{1,q}$ for $q\in [1,\infty]$ is right in Theorem \ref{theorem, embedding between Wiener and hp}.
Thus, Theorems \ref{theorem, embedding between Wiener and hp} and \ref{theorem, embedding between Wiener and L1 or L infty} are the essential improvement and extension of Theorem B, even when $1\leqslant p, q\leqslant \infty$.
\end{remark}

\begin{remark}\label{rek3}It will be also interesting to compare Theorem \ref{theorem, embedding between Wiener and Besov} with
the following results from \cite{Sugimoto_Tomita, Toft_Continunity, WH_JDE_2007}:
If $0<p\leqslant \infty$, $0\leqslant q\leqslant \infty$, then
\begin{equation}
  M_{p,q}^{\alpha(p,q)}\subset B_{p,q}~\text{for}~1/p<1/q;\hspace{6mm}B_{p,q}\subset M_{p,q}^{\beta(p,q)}~\text{for}~1/p>1/q.
\end{equation}
Recalling $M_{p, \min\{p,q\}}^s \subset W_{p,q}^s\subset M_{p, \max\{p,q\}}^s$, one may ask that
whether the above inclusion relations still hold if we replace the modulation space by corresponding Wiener amalgam space.
Theorem \ref{theorem, embedding between Wiener and Besov} answers this question.
More exactly, we verify the following inclusion relations
\begin{equation}
  W_{p,q}^{\alpha(p,q)}\subset B_{p,q}~\text{for}~1/p<1/q;\hspace{6mm}B_{p,q}\subset W_{p,q}^{\beta(p,q)}~\text{for}~1/p>1/q
\end{equation}
are negative for $0<p\leqslant \infty$, $0< q\leqslant \infty$.

We now turn to compare Theorem \ref{theorem, embedding between Wiener and hp} with the following known inclusion relations obtained in \cite{Lcal hardy and modulation_studia, sobolev and modulation}:
If $0<p<\infty$, $0\leqslant q\leqslant \infty$, then
\begin{equation}\label{introduction, 1}
  M_{p,q}^{\alpha(p,q)}\subset h_p~\text{for}~1/p<1/q;\hspace{6mm}h_p\subset M_{p,q}^{\beta(p,q)}~\text{for}~1/p>1/q.
\end{equation}
In Theorem \ref{theorem, embedding between Wiener and hp}, we obtain that
\begin{equation}\label{introduction, 2}
  W_{p,q}^{\alpha(p,q)}\subset h_p~\text{for}~1/p<1/q;\hspace{6mm}h_p\subset W_{p,q}^{\beta(p,q)}~\text{for}~1/p>1/q.
\end{equation}
Since $M_{p, \min\{p,q\}}^s \subset W_{p,q}^s\subset M_{p, \max\{p,q\}}^s$, the new inclusion relation (\ref{introduction, 2})
improves the known result (\ref{introduction, 1}).
\end{remark}

Finally, we give some useful corollaries which can be deduced directly by Theorems \ref{theorem, embedding between Wiener and hp}--\ref{theorem, embedding between Wiener and L1 or L infty}
and certain mild characterizations of the inclusion relations between $W_{p,q}^s$ and $L_p$ (see Section 3).
\begin{corollary}[Weighted Hausdorff-Young inequality]\label{corollary, weighted Hausdorff-Young inequality}
Let $1\leqslant p, q\leqslant \infty$, $s\in \mathbb{R}$. Let $R>0$ be a positive number. Then
\begin{enumerate}
 \item
  $\|f\|_{L_p}\lesssim \|\widehat{f}\|_{L_q^s}$ holds for all $f$ supported  on $B(0,R)$
  if and only if
  $s\geqslant \alpha(p,q)$
  with strict inequality when
  $1/q<1/p\wedge 1/2$ $(p\neq \infty)$ or $1/q<1$ $(p=\infty)$.
  \item
  $\|\widehat{f}\|_{L_q^s}\lesssim \|f\|_{L_p}$ holds for all $f$ supported  on $B(0,R)$
  if and only if
  $s\leqslant \beta(p,q)$
  with strict inequality when
  $1/q> 1/p\vee 1/2$ $(p\neq 1)$ or $q\neq \infty$ $(p=1)$.
\end{enumerate}
\end{corollary}

\begin{corollary}[Inequality for Fourier series]\label{corollary, Inequality for Fourier series}
Let $1\leqslant p, q\leqslant \infty$, $s\in \mathbb{R}$. Then
\begin{enumerate}
 \item
  $\|\sum_{k\in \mathbb{Z}^n}a_ke^{2\pi ikx}\|_{L_p(\mathbb{T}^n)} \lesssim \|\{a_k\}_{k\in \mathbb{Z}^n}\|_{l_q^{n(1/2-1/q),0}}$
  holds for all truncated sequences $\vec{a}=\{a_k\}_{k\in \mathbb{Z}^n}$
  if and only if
  $s\geqslant \alpha(p,q)$
  with strict inequality when
  $1/q<1/p\wedge 1/2$ $(p\neq \infty)$ or $1/q<1$ $(p=\infty)$.
  \item
  $\|\{a_k\}_{k\in \mathbb{Z}^n}\|_{l_q^{n(1/2-1/q),0}} \lesssim \|\sum_{k\in \mathbb{Z}^n}a_ke^{2\pi ikx}\|_{L_p(\mathbb{T}^n)}$
  holds for all truncated sequences $\vec{a}=\{a_k\}_{k\in \mathbb{Z}^n}$
  if and only if
  $s\leqslant \beta(p,q)$
  with strict inequality when
  $1/q> 1/p\vee 1/2$ $(p\neq 1)$ or $q\neq \infty$ $(p=1)$.
\end{enumerate}
\end{corollary}

This paper is organized as follows. In Section 2, we will give some basic notations,
and recall some definitions and basic properties for the function spaces used throughout this paper.
Some preliminary lemmas will also be given in this section. Section 3 is devoted
to some mild characterizations of inclusion relations between Wiener amalgam and Triebel-Lizorkin spaces.
Our main theorems will be proved in Sections 4-6.

\section{PRELIMINARIES}

We recall some notations. Let $C$ be a positive constant that may depend on $n,p_i,q_i,s_i,\alpha$ for $i=1, 2.$
The notation $X\lesssim  Y$ denotes the statement that $X\leqslant CY$,
the notation $X\sim Y$ means the statement $X\lesssim Y \lesssim X$,
and the notation $X\simeq Y$ denotes the statement $X=CY$.
For a multi-index $k=(k_1,k_2,...k_n)\in \mathbb{Z}^{n}$,
we denote $|k|_{\infty}: =\mathop{\sup}_{i=1,2...n}|k_i|$, and $\langle k\rangle: =(1+|k|^{2})^{1/2}.$

In this paper, for the sake of simplicity, we use the notation $"\mathscr{L}"$ to denote some large positive number which may be changed
corresponding to the exact environment.

Let $\mathscr {S}:= \mathscr {S}(\mathbb{R}^{n})$ be the Schwartz space
and $\mathscr {S}':=\mathscr {S}'(\mathbb{R}^{n})$ be the space of tempered distributions.
We define the Fourier transform $\mathscr {F}f$ and the inverse Fourier transform $\mathscr {F}^{-1}f$  of $f\in \mathscr {S}(\mathbb{R}^{n})$ by
$$
\mathscr {F}f(\xi)=\hat{f}(\xi)=\int_{\mathbb{R}^{n}}f(x)e^{-2\pi ix\cdot \xi}dx
~~
,
~~
\mathscr {F}^{-1}f(x)=\hat{f}(-x)=\int_{\mathbb{R}^{n}}f(\xi)e^{2\pi ix\cdot \xi}d\xi.
$$

We recall some definitions of the function spaces treated in this paper.
\begin{definition}
Let $0<p \leqslant \infty$, $s\in \mathbb{R}$. The weighted Lebesgue space $L_{x, p}^s$ consists of all measurable functions $f$ such that
\begin{numcases}{\|f\|_{L_{x, p}^s}=}
     \left(\int_{\mathbb{R}^n}|f(x)|^p \langle x\rangle^{ps} dx\right)^{{1}/{p}}, &$p<\infty$   \\
     ess\sup_{x\in \mathbb{R}^n}|f(x)\langle x\rangle^s|,  &$p=\infty$
\end{numcases}
is finite.
If $f$ is defined on $\mathbb{Z}^n$, we denote
\begin{numcases}{\|f\|_{l_{k,p}^{s,0}}=}
\left(\sum_{k\in \mathbb{Z}^n}|f(k)|^p \langle k\rangle^{ps}\right)^{{1}/{p}}, &$p<\infty$
\\
\sup_{k\in \mathbb{Z}^n}|f(k)\langle k\rangle^s|,\hspace{15mm} &$p=\infty$
\end{numcases}
and $l_{k,p}^s$ as the (quasi) Banach space of functions $f: \mathbb{Z}^n\rightarrow \mathbb{C}$ whose $l_{k,p}^s$ norm is finite.
If $f$ is defined on $\mathbb{N}$, we denote
\begin{numcases}{\|f\|_{l_{j,p}^{s,1}}=}
\left(\sum_{j\in \mathbb{N}}2^{jsp}|f(j)|^p\right)^{{1}/{p}}, &$p<\infty$
\\
\sup_{j\in \mathbb{N}}|2^{js}f(j)|,\hspace{15mm} &$p=\infty$
\end{numcases}
and $l_{j,p}^{s,1}$ as the (quasi) Banach space of functions $f: \mathbb{N}\rightarrow \mathbb{C}$ whose $l_{j,p}^{s,1}$ norm is finite.
We write $L_p^s$, $l_p^{s,0}$, $l_p^{s,1}$ for short, respectively, if there is no confusion.
\end{definition}

The translation operator is defined as $T_{x_0}f(x)=f(x-x_0)$ and
the modulation operator is defined as $M_{\xi}f(x)=e^{2\pi i\xi \cdot x}f(x)$, for $x$, $x_0$, $\xi\in\mathbb{R}^n$.
Fixed a nonzero function $\phi\in \mathscr{S}$, the short-time Fourier
transform of $f\in \mathscr{S}'$ with respect to the window $\phi$ is given by
\begin{equation}
V_{\phi}f(x,\xi)=\langle f,M_{\xi}T_x\phi\rangle=\int_{\mathbb{R}^n}f(y)\overline{\phi(y-x)}e^{-2\pi iy\cdot \xi}dy.
\end{equation}

Now, we give the definitions of modulation and Wiener amalgam spaces.
\begin{definition}\label{Definition, modulation space, continuous form}
Let $0<p, q\leqslant \infty$, $s\in \mathbb{R}$.
Given a window function $\phi\in \mathscr{S}\backslash\{0\}$, the modulation space $M_{p,q}^s$ consists
of all $f\in \mathscr{S}'(\mathbb{R}^n)$ such that the norm
\begin{equation}
\begin{split}
\|f\|_{M_{p,q}^s}&=\big\|\|V_{\phi}f(x,\xi)\|_{L_{x,p}}\big\|_{L_{\xi,q}^s}
\\&
=\left(\int_{\mathbb{R}^n}\left(\int_{\mathbb{R}^n}|V_{\phi}f(x,\xi)|^{p} dx\right)^{{q}/{p}}\langle\xi\rangle^{sq}d\xi\right)^{{1}/{q}}
\end{split}
\end{equation}
is finite, with the usual modification when $p=\infty$ or $q=\infty$.
In addition, we write $M_{p,q}:=M_{p,q}^{0}$.
\end{definition}

The above definition of \ $M_{p,q}^s$ \  is independent of the choice of window function $\phi$.
One can see this fact in \cite{Grochnig} \ for the case $(p,q)\in\lbrack 1,\infty ]^{2}$,
and in \cite{Galperin_Samarah} for the case $(p,q)\in (0,\infty ]^{2}\backslash\lbrack 1,\infty]^{2}$.
More properties of modulation spaces can be founded in \cite{Feichtinger, Wang_book}.
One can also see \cite{Feichtinger_Survey} for a survey of modulation spaces.

\begin{definition}\label{Definition, Wiener amalgam space, continuous form}
Let $0<p, q\leqslant \infty$, $s\in \mathbb{R}$.
Given a window function $\phi\in \mathscr{S}\backslash\{0\}$, the Wiener amalgam space $W_{p,q}^s$ consists
of all $f\in \mathscr{S}'(\mathbb{R}^n)$ such that the norm
\begin{equation}
\begin{split}
\|f\|_{W_{p,q}^{s}}&=\big\|\|V_{\phi}f(x,\xi)\|_{L_{\xi, q}^s}\big\|_{L_{x,p}}
\\&
=\left(\int_{\mathbb{R}^n}\left(\int_{\mathbb{R}^n}|V_{\phi}f(x,\xi)|^{q}\langle \xi\rangle^{sq}d\xi\right)^{{p}/{q}}dx\right)^{{1}/{p}}
\end{split}
\end{equation}
is finite, with the usual modifications when $p=\infty$ or $q=\infty$.
\end{definition}

We recall an embedding lemma as follows. One can see more basic properties about Wiener amalgam spaces in \cite{Cordero_Nicola_Sharpness, Guo_Characterization, C.Heil}.
\begin{lemma}\label{lemma, embedding for Wiener amalgam space, change of variable}
Let $0<p_i, q_i\leqslant \infty$, $s_i\in \mathbb{R}$ for $i=1, 2$. Then
\begin{equation}
  W_{p_1,q_1}^{s_1}\subset W_{p_2,q_2}^{s_2}
\end{equation}
holds if and only if
\begin{equation}
\begin{cases}
s_2\leqslant s_1  \\
\frac{1}{p_2}\leqslant \frac{1}{p_1}\\
\frac{1}{q_2}+\frac{s_2}{n}< \frac{1}{q_1}+\frac{s_1}{n}
\end{cases}
\text{or} \hspace{10mm}
\begin{cases}
s_2=s_1\\
\frac{1}{p_2}\leqslant \frac{1}{p_1}\\
q_2=q_1.
\end{cases}
\end{equation}
\end{lemma}

The following lemma provides some inclusion relations between modulation and Wiener amalgam spaces.
One can verify it by using Minkowski's inequality.

\begin{lemma}\label{lemma, embedding between modulation and wiener}
Let $0<p, q\leqslant \infty$ and $s\in \mathbb{R}$.
We have
\begin{equation}
  M_{p, \min\{p,q\}}^s \subset W_{p,q}^s\subset M_{p, \max\{p,q\}}^s.
\end{equation}
\end{lemma}

Next, we recall some function spaces associated with the dyadic decomposition of $\mathbb{R}^{n}$.
Let $\varphi(\xi)$ be a smooth bump function supported in the ball $\{\xi: |\xi|<3/2\}$ and be equal to 1 on the ball $\{\xi: |\xi|\leqslant 4/3\}$.
Denote
\begin{equation}
\psi(\xi)=\varphi(\xi)-\varphi(2\xi),
\end{equation}
and a function sequence
\begin{equation}\label{introduction, 3}
\begin{cases}
\psi_j(\xi)=\psi(2^{-j}\xi),~j\in \mathbb{Z}^{+},
\\
\psi_0(\xi)=1-\sum_{j\in \mathbb{Z}^+}\psi_j(\xi)=\varphi(\xi).
\end{cases}
\end{equation}
For $j\in \mathbb{N}=\mathbb{Z}^{+} \cup \{0\}$, we define the Littlewood-Paley operators
\begin{equation}
\Delta_j=\mathscr{F}^{-1}\psi_j\mathscr{F}.
\end{equation}
Let $0< p,q\leqslant\infty$ and $s\in \mathbb{R}$. For $f\in\mathscr {S}'$, set
\begin{equation}
\|f\|_{B_{p,q}^s}=\left(\sum_{j=0}^{\infty}2^{jsq}\|\Delta_jf\|_{L_p}^q \right)^{1/q}.
\end{equation}
The (inhomogeneous) Besov space is the space of all tempered distributions $f$ for which the quantity $\|f\|_{B_{p,q}^s}$ is finite.
We write $B_{p,q}:=B_{p,q}^0$ for short.

Let $0<p<\infty$, $0< q\leqslant\infty$ and $s\in \mathbb{R}$. For $f\in\mathscr {S}'$, set
\begin{equation}
\|f\|_{F_{p,q}^s}=\left\|\left(\sum_{j=0}^{\infty}2^{jsq}|\Delta_jf|^q \right)^{1/q}\right\|_{L_p}.
\end{equation}
The (inhomogeneous) Triebel-Lizorkin space is the space of all tempered distributions $f$ for which the quantity $\|f\|_{F_{p,q}^s}$ is finite.
We write $F_{p,q}:=F_{p,q}^0$ for short.

Now, we turn to introduce the local Hardy space of Goldberg \cite{Goldberg}.
Let $\psi\in\mathscr{S}$ satisfy $\int_{\mathbb{R}^n}\psi(x)dx\neq0$.
Define $\psi_t=t^{-n}\psi(x/t)$.
The local Hardy spaces is defined by
\begin{equation*}
  h^p:=\{f\in\mathscr{S}': \|f\|_{h^p}=\|\sup_{0<t<1}|\psi_t\ast f|\|_{L^p}<\infty\}.
\end{equation*}
We note that the definition of the local Hardy spaces is independent of the choice of $\psi\in \mathscr{S}$.
We also remark that the local $h_p$ is equivalent with the inhomogeneous Triebel-Lizorkin space $F_{p,2}$ (see 1.4 in \cite{Tribel_92}).
A function $a$ is called a small $h_p$-atom if
\begin{equation}
  \begin{split}
    &\text{supp}a\subset Q,~ \|a\|_{L_{\infty}}\leqslant |Q|^{-1/p},
    \\
    &\int_{\mathbb{R}^n}x^{\gamma}a(x)\ for\ all\ |\gamma|\leqslant [n(1/p-1)],
  \end{split}
\end{equation}
where $Q$ is a cube with $|Q|<1$, $|Q|$ is the Lebesgue measure of $Q$, and $[n(1/p-1)]$ is the integer part of $n(1/p-1)$.
A function $a$ on $\mathbb{R}^n$ is called a big $h_p$-atom if there exists a cube $Q$ with $|Q|\geqslant 1$, such that
\begin{equation}
\text{supp}a\subset Q, \|a\|_{L_{\infty}}\leqslant |Q|^{-1/p}.
\end{equation}
All the big and small-$h_p$ atoms are collectively called $h_p$-atom.
We recall $\|a\|_{h_p}\lesssim 1$ for all $h_p$-atoms. On the other hand, any $f\in h_p$ can be represented by
\begin{equation}
  f=\sum_{j=1}^{\infty}\lambda_ja_j,
\end{equation}
where the series converges in the sense of distribution, $\{a_j\}$ is a collection of $h_p$-atoms and $\{\lambda_j\}$ is a sequence of complex numbers such that
\begin{equation}
  \bigg(\sum_{j=1}^{\infty} \lambda_j^p\bigg)^{1/p}\lesssim \|f\|_{h_p}.
\end{equation}

We recall two convolution inequalities for continuous and discrete form respectively.
\begin{lemma}[{Weighted convolution in $L_{p}$ with $p<1$}]\label{lemma, convolution for p<1}
Let $0<p< 1$, $s\in \mathbb{R}$ and $B(x_0,R)=\{x: |x-x_0|\le R\}$.
Suppose $f, g\in L^p$ with Fourier support in $B(x_0,R)$ and $B(x_1,R)$ respectively.
Then there exists a constant $C_R>0$ which is independent of $x_0, x_1$ such that
$$
\||f|*|g|\|_{L_p^s} \le C_R \|f\|_{L_p^s} \|g\|_{L_p^{|s|}}.
$$
\end{lemma}
\begin{proof}
This lemma can be found in \cite{Triebel_book_1983} when $s=0$.
By a standard limiting argument, we only need to state the proof for $f,g \in \mathscr{S}$.
In this case, we observe that $H(y)=f(x-y)g(y)$ is a Schwartz function with Fourier support in $B(x_1-x_0, 2R)$ for every $x\in \mathbb{R}^n$.
By the compactness of Fourier support of $H$, we have
\begin{equation}
\int_{\mathbb{R}^n}|f(x-y)g(y)|dy\lesssim_R \left(\int_{\mathbb{R}^n}|f(x-y)g(y)|^pdy\right)^{1/p}.
\end{equation}
Observing $\langle x\rangle^s \lesssim \langle x-y\rangle^s\langle y\rangle^{|s|}$, we deduce
\begin{equation}
\left(\int_{\mathbb{R}^n}|f(x-y)g(y)|dy\right)\langle x\rangle^s\lesssim \left(\int_{\mathbb{R}^n}|f(x-y)\langle x-y\rangle^s g(y)\langle y\rangle^{|s|}|^pdy\right)^{1/p}.
\end{equation}
Taking the $L_{x;p}$-norm on both sides of the above inequality, we then use the Fubini theorem to obtain the desired conclusion.
\end{proof}

\begin{lemma}[{Weighted convolution in $l_{p}$}]\label{lemma, convolution for lp}
Let $0<p\leqslant \infty$, $s\in \mathbb{R}$. We have
\begin{equation}
  l_p^s\ast l_{p \wedge  1}^{|s|} \subset l_p^s.
\end{equation}
\end{lemma}
This lemma is a special version of discrete weighted convolution inequality, which can be deduced without difficult.
The reader can find its general version in \cite{Guo_Characterization}.

Next, we recall a probability inequality. One can find its proof in \cite{Gut_book}.
\begin{lemma}[Khinchin's inequality, see \cite{Gut_book}]
  Let $0<p<\infty$, $\{\omega_k\}_{k=1}^N$ be a sequence of independent random variables taking values $\pm 1$ with equal probability. Denote expectation (integral over the probability space) by $\mathbb{E}$.
  For any sequence of complex numbers $\{a_k\}_{k=1}^N$, we have
  \begin{equation}
    \mathbb{E}(|\sum_{k=1}^Na_k\omega_k|^p)\sim \left(\sum_{k=1}^N|a_k|^2\right)^{\frac{p}{2}},
  \end{equation}
  where the implicit constants depend on $p$ only.
\end{lemma}

Finally, we recall two embedding lemmas for weighted sequence spaces.

\begin{lemma}[Sharpness of embedding, for uniform decomposition] \label{lemma, Sharpness of embedding, for uniform decomposition}
Suppose $0<q_1,q_2\leqslant \infty$, $s_1,s_2\in \mathbb{R}$. Then
\begin{equation}
l_{q_1}^{s_1,0}\subset l_{q_2}^{s_2,0}
\end{equation}
holds if and only if
\begin{equation}
\begin{cases}
s_2\leqslant s_1  \\
\frac{1}{q_2}+\frac{s_2}{n}< \frac{1}{q_1}+\frac{s_1}{n}
\end{cases}
\text{or} \hspace{10mm}
\begin{cases}
s_2=s_1\\
q_2=q_1.
\end{cases}
\end{equation}
\end{lemma}

\begin{lemma}[Sharpness of embedding, for dyadic decomposition] \label{lemma, Sharpness of embedding, for dyadic decomposition}
Suppose $0<q_1, q_2\leqslant \infty$, $s_1, s_2\in \mathbb{R}$. Then
\begin{equation}
l_{q_1}^{s_1,1}\subset l_{q_2}^{s_2,1}
\end{equation}
holds if and only if
\begin{equation}
s_2<s_1
\hspace{10mm}
\text{or} \hspace{10mm}
\begin{cases}
s_2=s_1,\\
1/q_2\leqslant 1/q_1.
\end{cases}
\end{equation}
\end{lemma}

\section{Characterizations by localization and discretization}
This section is devoted to some mild characterizations of the inclusion relations between Wiener amalgam and Triebel-Lizorkin spaces.
We start from recalling a localization principle for the classical Triebel-Lizorkin space $F_{p,q}^s$.
\begin{lemma}[Localization principle for Triebel-Lizorkin spaces, cf. \cite{Triebel_book_2010}]\label{lemma, Localization principle for Triebel-Lizorkin space}
Let $0<p<\infty$, $0<q\leqslant \infty$, $s\in \mathbb{R}$.
Let $\{\psi_k\}_{k\in \mathbb{Z}^n}$ be a smooth uniform partition of unity such that
$\sum_{k\in \mathbb{Z}^n}\psi_k(x)=1$,
where $\psi_k(x)=\psi(x-k)$, $\psi$ is a smooth function with compact support near the origin.
Then
\begin{equation}
\left(\sum_{k\in \mathbb{Z}^n}\|\psi_kf\|^p_{F_{p,q}^s}\right)^{1/p}
\end{equation}
is an equivalent quasi-norm in $F_{p,q}^s$.
\end{lemma}

Next, we give a similar localization result for the Wiener amalgam space $W_{p,q}^s$.
Due to the symmetry of time-frequency, localization principle in Wiener amalgam spaces is more convenient to establish than which in Triebel-Lizorkin spaces.

\begin{proposition}[Localization principle for Wiener amalgam spaces]\label{lemma, Localization principle for Wiener amalgam space}
Let $0<p, q\leqslant \infty$, $s\in \mathbb{R}$.
Let $\{\psi_k\}_{k\in \mathbb{Z}^n}$ be a smooth uniform partition of unity such that
$\sum_{k\in \mathbb{Z}^n}\psi_k(x)=1$,
where $\psi_k(x)=\psi(x-k)$, $\psi$ is a smooth function with compact support near the origin.
Then
\begin{equation}
\left(\sum_{k\in \mathbb{Z}^n}\|\psi_kf\|^p_{W_{p,q}^s}\right)^{1/p}
\end{equation}
with usual modification for $p=\infty$ is an equivalent quasi-norm in $W_{p,q}^s$.
\end{proposition}
\begin{proof}
We only give the proof for $p, q<\infty$, since the other cases can be handled similarly.
Take window function $\phi$ to be a smooth function with compact support near the origin.
By the almost orthogonality of $\{V_{\phi}(\psi_kf)(x,\xi)\}_{k\in \mathbb{Z}^n}$, we have
\begin{equation}
|V_{\phi}f(x,\xi)|^q=|\sum_{k\in \mathbb{Z}^n}V_{\phi}(\psi_kf)(x,\xi)|^q\lesssim\sum_{k\in \mathbb{Z}^n}|V_{\phi}(\psi_kf)(x,\xi)|^q.
\end{equation}
Then
\begin{equation}
  \begin{split}
    \|V_{\phi}f(x,\xi)\|^q_{L^s_{\xi,q}}
    =&
    \int_{\mathbb{R}^n}|V_{\phi}f(x,\xi)|^q\langle \xi\rangle^{sq}d\xi
    \\
    \lesssim&
    \sum_{k\in \mathbb{Z}^n}\int_{\mathbb{R}^n}|V_{\phi}(\psi_kf)(x,\xi)|^q\langle \xi\rangle^{sq}d\xi
    \\
    =&
    \sum_{k\in \mathbb{Z}^n}\|V_{\phi}(\psi_kf)(x,\xi)\|^q_{L^s_{\xi,q}}.
  \end{split}
\end{equation}
Combining with the above inequality, we use the almost orthogonality of $\{\|V_{\phi}(\psi_kf)(x,\xi)\|_{L^s_{\xi,q}}\}$ to deduce
\begin{equation}
  \begin{split}
    \|f\|^p_{W_{p,q}^s}
    =&
    \big\|\|V_{\phi}f(x,\xi)\|_{L_{\xi, q}^s}\big\|^p_{L_{x,p}}
    \\
    \lesssim&
    \int_{\mathbb{R}^n}\left(\sum_{k\in \mathbb{Z}^n}\|V_{\phi}(\psi_kf)(x,\xi)\|^q_{L^s_{\xi,q}}\right)^{p/q}dx
    \\
    \lesssim &
    \sum_{k\in \mathbb{Z}^n}\int_{\mathbb{R}^n}\|V_{\phi}(\psi_kf)(x,\xi)\|^p_{L^s_{\xi,q}}dx
    \\
    =&
    \sum_{k\in \mathbb{Z}^n}\|\psi_kf\|^p_{W_{p,q}^s},
  \end{split}
\end{equation}
which implies $\|f\|_{W_{p,q}^s}\lesssim \left(\sum_{k\in \mathbb{Z}^n}\|\psi_kf\|^p_{W_{p,q}^s}\right)^{1/p}.$

Now, we turn to the proof of the opposite direction. By the definition of short-time Fourier transform, we have
\begin{equation}
  \begin{split}
      V_{\phi}(\psi_kf)(x,\xi)
      =&
      \mathscr{F}\big(\psi_k(\cdot)f(\cdot)\phi(\cdot-x)\big)(\xi)
      \\
      =&
      \big((\mathscr{F}\psi_k)(\cdot)\ast V_{\phi}(f)(x,\cdot)\big)(\xi).
  \end{split}
\end{equation}
Then, we use Young's inequality or Lemma \ref{lemma, convolution for p<1} to deduce
\begin{equation}
  \begin{split}
      \|V_{\phi}(\psi_kf)(x,\xi)\|_{L_{\xi,q}^s}
      =&
      \|\big((\mathscr{F}\psi_k)(\cdot)\ast V_{\phi}(f)(x,\cdot)\big)(\xi)\|_{L_{\xi,q}^s}
      \\
      \lesssim &
      \|V_{\phi}(f)(x,\xi)\|_{L_{\xi,q}^s}\|\mathscr{F}\psi_k\|_{L_{q\wedge 1}^{|s|}}
      \\
      \lesssim &
      \|V_{\phi}(f)(x,\xi)\|_{L_{\xi,q}^s}\|\mathscr{F}\psi\|_{L_{q\wedge 1}^{|s|}}
      \\
      \lesssim &
      \|V_{\phi}(f)(x,\xi)\|_{L_{\xi,q}^s}.
  \end{split}
\end{equation}
Observing that the support of $\|V_{\phi}(\psi_kf)(\cdot,\xi)\|_{L_{\xi,q}^s}$ is near $k$, we choose a positive number $r$ such that
$\text{supp}(\|V_{\phi}(\psi_kf)(\cdot,\xi)\|_{L_{\xi,q}^s})\subset B(k,r)$.
Thus,
\begin{equation}
  \begin{split}
    \sum_{k\in \mathbb{Z}^n}\|\psi_kf\|^p_{W_{p,q}^s}
    =&
    \sum_{k\in \mathbb{Z}^n}\int_{B(k,r)}\|V_{\phi}(\psi_kf)(x,\xi)\|^p_{L_{\xi,q}^s}dx
    \\
    \lesssim &
    \sum_{k\in \mathbb{Z}^n}\int_{B(k,r)}\|V_{\phi}(f)(x,\xi)\|^p_{L_{\xi,q}^s}dx
    \\
    \lesssim &
    \int_{\mathbb{R}^n}\|V_{\phi}(f)(x,\xi)\|^p_{L_{\xi,q}^s}dx=\|f\|^p_{W_{p,q}^s},
  \end{split}
\end{equation}
which implies $\left(\sum_{k\in \mathbb{Z}^n}\|\psi_kf\|^p_{W_{p,q}^s}\right)^{1/p}\lesssim \|f\|_{W_{p,q}^s}.$
\end{proof}

We give the following lemma for the local property of $W_{p,q}^s$. The reader can find an unweighted result in \cite{change of variable}.
\begin{lemma}\label{lemma, equivalent norm, compact support}
Let $0<p, q\leqslant \infty$, $s\in \mathbb{R}$,
$K$ be a compact subset of $\mathbb{R}^n$, and $f$ be a tempered distribution with compact support $K$.
Then $f\in W_{p,q}^s$ if and only if $\widehat{f}\in L_{q, s}$, and
\begin{equation}
\|f\|_{W_{p,q}^s}\sim_K  \|\widehat{f}\|_{L_{q,s}}.
\end{equation}
\end{lemma}
\begin{proof}
We only state the proof for $p<\infty$, the case $p=\infty$ is similar.
Without loss of generality, we assume $\text{supp}f\in B(0, N)$, where $N$ is the diameter of $K$.
Choose window function $\phi$ be a smooth function with compact support such that $\phi=1$ in $B(0,2N)$.
We have
\begin{equation}
  \begin{split}
    \int_{\mathbb{R}^n}\|\mathscr{F}(f(\cdot)\overline{\phi}(\cdot-x))(\xi)\|^p_{L_{\xi,q}^s}dx
    \gtrsim &
    \int_{B(0,N)}\|\mathscr{F}(f(\cdot)\overline{\phi}(\cdot-x))(\xi)\|^p_{L_{\xi,q}^s}dx
    \\
    \gtrsim &
    \int_{B(0,N)}\|\mathscr{F}(f(\cdot))(\xi)\|^p_{L_{\xi,q}^s}dx
    \\
    \gtrsim &
    \|\widehat{f}\|^p_{L_q^s},
  \end{split}
\end{equation}
which implies $\|\widehat{f}\|_{L_q^s}\lesssim \|f\|_{W_{p,q}^s}$.

On the other hand, by the support of $f$ and $\phi$, we have
\begin{equation}
  \int_{\mathbb{R}^n}\|\mathscr{F}(f(\cdot)\overline{\phi}(\cdot-x))(\xi)\|^p_{L_{\xi,q}^s}dx
  \lesssim
  \int_{B(0,3N)}\|\mathscr{F}(f(\cdot)\overline{\phi}(\cdot-x))(\xi)\|^p_{L_{\xi,q}^s}dx.
\end{equation}
Observing $|\mathscr{F}(f(\cdot)\overline{\phi}(\cdot-x))(\xi)|\leqslant (|\widehat{f}|\ast |\widehat{\overline{\phi}}|)(\xi)$,
we use Young's inequality or Lemma \ref{lemma, convolution for p<1} to deduce that
\begin{equation}
  \|(|\widehat{f}|\ast |\widehat{\overline{\phi}}|)(\xi)\|_{L_q^s}\lesssim \|\widehat{f}\|_{L_q^s}.
\end{equation}
Thus, we have
\begin{equation}
  \int_{\mathbb{R}^n}\|\mathscr{F}(f(\cdot)\overline{\phi}(\cdot-x))(\xi)\|^p_{L_{\xi,q}^s}dx
  \lesssim
  \int_{B(0,3N)}\|\widehat{f}\|_{L_q^s}^pdx
  \lesssim \|\widehat{f}\|^p_{L_q^s},
\end{equation}
which implies that $\|f\|_{W_{p,q}^s}\lesssim \|\widehat{f}\|_{L_q^s}.$
\end{proof}

By the localization technique, we give a mild characterization for the embedding between Triebel-Lizorkin and Wiener amalgam spaces.
\begin{proposition}[Mild characterization of the embedding between Triebel-Lizorkin and Wiener amalgam spaces]\label{proposition, mild characterization}
Let $0<p_1, q_1, q_2\leqslant \infty$, $0<p_2< \infty$, $s\in \mathbb{R}$. Take $r$ to be a fixed positive real number,
we have the following statements:
\begin{enumerate}
 \item
  $W_{p_1,q_1}^s\subset F_{p_2,q_2}$
  if and only if $1/p_2\leqslant 1/p_1$ and one of the following statements (a) and
  (b) holds:
    \begin{enumerate}
     \item
     $\|f\|_{F_{p_2,q_2}}\lesssim \|f\|_{W_{p_1,q_1}^s}$ for any $f\in \mathscr{S}'$ with support in $B(0, r)$,
     \item
     $\|f\|_{F_{p_2,q_2}}\lesssim \|\hat{f}\|_{L_{q_1}^s}$ for any $f\in \mathscr{S}'$ with support in $B(0, r)$.
   \end{enumerate}
 \item
  $F_{p_2,q_2}\subset W_{p_1,q_1}^s$
  if and only if $1/p_1\leqslant 1/p_2$ and one of the following statements (a) and
  (b) holds:
    \begin{enumerate}
     \item
     $\|f\|_{W_{p_1,q_1}^s}\lesssim \|f\|_{F_{p_2,q_2}}$ for any $f\in \mathscr{S}'$ with support in $B(0, r)$,
     \item
     $\|\hat{f}\|_{L_{q_1}^s}\lesssim\|f\|_{F_{p_2,q_2}}$ for any $f\in \mathscr{S}'$ with support in $B(0, r)$.
   \end{enumerate}
\end{enumerate}
\end{proposition}
\begin{proof}
Since the symmetry between statements (1) and (2), we only give the proof for the statement (1).
Obviously, $W_{p_1,q_1}^s\subset F_{p_2,q_2}$ implies (a), and then implies (b) by Lemma \ref{lemma, equivalent norm, compact support}.
Take $h$ to be a smooth function with small Fourier support near the origin, such that $\Delta_0h=h$.
Denote $\widehat{h_{\epsilon}}(\xi)=\widehat{h}(\xi/\epsilon)$. We obtain
\begin{equation}
  \|h_{\epsilon}\|_{F_{p_2,q_2}}\sim \|h_{\epsilon}\|_{L_{p_2}}\sim \epsilon^{n(1-1/p_2)}
\end{equation}
for $\epsilon \in (0,1)$.
On the other hand, by Lemma \ref{lemma, embedding between modulation and wiener}, we have
\begin{equation}
  \|h_{\epsilon}\|_{M_{p_1,\min\{p_1,q_1\}}^{s}}\lesssim \|h_{\epsilon}\|_{W_{p_1,q_1}^s}\lesssim \|h_{\epsilon}\|_{M_{p_1,\max\{p_1,q_1\}}^{s}}.
\end{equation}
Moreover, we have $\|h_{\epsilon}\|_{M_{p_1,\min\{p_1,q_1\}}^{s}}\sim \|h_{\epsilon}\|_{L_{p_1}}$ and $\|h_{\epsilon}\|_{M_{p_1,\max\{p_1,q_1\}}^{s_2}}\sim \|h_{\epsilon}\|_{L_{p_1}}$
by the local property of modulation spaces (see Lemma 2.3 in \cite{Guo_Characterization}, for instance).
Thus, we have
\begin{equation}
  \|h_{\epsilon}\|_{W_{p_1,q_1}^s}\sim \|h_{\epsilon}\|_{L_{p_1}}\sim \epsilon^{n(1-1/p_1)}
\end{equation}
for $\epsilon \in (0,1)$.
So $W_{p_1,q_1}^s\subset F_{p_2,q_2}$ implies $\epsilon^{n(1-1/p_2)} \lesssim \epsilon^{n(1-1/p_1)}$ for $\epsilon \in (0,1)$.
Letting $\epsilon \rightarrow 0$, we obtain $1/p_2\leqslant 1/p_1$.

Next, we only need to verify that $1/p_2\leqslant 1/p_1$ and (a) implies $W_{p_1,q_1}^s\subset F_{p_2,q_2}$.
We first verify the conclusion for $r=2\sqrt{n}$, then the general $r$ will be verified by a dilation argument.
Let $\psi$ be smooth function with compact support included in $B(0, 2\sqrt{n})$ such that
$\sum_{k\in \mathbb{Z}^n}\psi_k(x)=1$,
where $\psi_k(x)=\psi(x-k)$.
For $f\in W_{p_1,q_1}^s$, we have $\|f\|_{W_{p_1,q_1}^s}\sim \left(\sum_{k\in \mathbb{Z}^n}\|\psi_kf\|^{p_1}_{W_{p_1,q_1}^s}\right)^{1/{p_1}}$ and $\psi_kf\in W_{p_1,q_1}^s$.
Thus, for every $k\in \mathbb{Z}^n$, $T_{-k}(\psi_kf)$ is an element of $W_{p_1,q_1}^{s}$ with compact support in $B(0, 2\sqrt{n})$.
By the assumption (a), we have that $T_{-k}(\psi_kf)$ is also a element of $F_{p_2,q_2}$ satisfying $\|T_{-k}(\psi_kf)\|_{F_{p_2,q_2}}\lesssim \|T_{-k}(\psi_kf)\|_{W_{p_1,q_1}^s}$.
By the translation invariant of norms $\|\cdot\|_{F_{p_2,q_2}}$ and $\|\cdot\|_{W_{p_1,q_1}^s}$, we deduce that
\begin{equation}
\|\psi_kf\|_{F_{p_2,q_2}}\lesssim \|\psi_kf\|_{W_{p_1,q_1}^s}.
\end{equation}
Taking $l^{p_1}$ on both sides of the above inequality and use the fact that $l_{p_1}\subset l_{p_2}$ for $1/p_2\leqslant 1/p_1$, we have
\begin{equation}
\left(\sum_{k\in \mathbb{Z}^n}\|\psi_kf\|^{p_2}_{F_{p_2,q_2}}\right)^{1/p_2}\lesssim \left(\sum_{k\in \mathbb{Z}^n}\|\psi_kf\|^{p_1}_{W_{p_1,q_1}^s}\right)^{1/p_1}\sim \|f\|_{W_{p_1,q_1}^s}.
\end{equation}
Next, we invoke a limiting argument to verify $f\in F_{p_2,q_2}$. Denote $f_m=\sum_{|k|\leqslant m}f\psi_k$, then $f_m\in h_p$ and
\begin{equation}
  \|f_m-f_l\|_{F_{p_2,q_2}}\lesssim \left(\sum_{l<|k|\leqslant m}\|\psi_kf\|^{p_2}_{F_{p_2,q_2}}\right)^{1/p_2} \rightarrow 0
\end{equation}
as $m>l\rightarrow \infty$.
We have that
$\{f_m\}_{m=1}^{\infty}$ forms a Cauchy sequence in $F_{p_2,q_2}$.
We also have $f_m\rightarrow f$ in $\mathscr{S}'$. Thus, $f\in F_{p_2,q_2}$, and
\begin{equation}
  \begin{split}
    \|f\|_{F_{p_2,q_2}}
    =&
    \lim_{m\rightarrow \infty}\|f_m\|_{F_{p_2,q_2}}
    \\
    \lesssim &
    \lim_{m\rightarrow \infty}\left(\sum_{|k|\leqslant m}\|\psi_kf\|^{p_2}_{F_{p_2,q_2}}\right)^{1/p_2}
    \\
    \lesssim &
    \lim_{m\rightarrow \infty}\left(\sum_{|k|\leqslant m}\|\psi_kf\|^{p_1}_{W_{p_1,q_1}^s}\right)^{1/p_1}\sim \|f\|_{W_{p_1,q_1}^s}.
  \end{split}
\end{equation}

Now, we finish this proof for $r=2\sqrt{n}$ and turn to the proof for general $r$.
Assume $\|g\|_{F_{p_2,q_2}}\lesssim \|g\|_{W_{p_1,q_1}^s}$ for any $g\in \mathscr{S}'$ supported in $B(0, r)$ with $r<2\sqrt{n}$.
For a fixed $f$ supported in $B(0,2\sqrt{n})$,
we denote $f_r(x)=f(2\sqrt{n}x/r)$. Obviously, $f_r$ has support in $B(0,r)$. Thus, we have
\begin{equation}
  \|f_r\|_{F_{p_2,q_2}}\lesssim \|f_r\|_{W_{p_1,q_1}^s}
\end{equation}
by the assumption.
By the dilation properties of Wiener amalgam and Triebel-Lizorkin spaces, we deduce
\begin{equation}
  \|f\|_{F_{p_2,q_2}}\lesssim_r \|f_r\|_{F_{p_2,q_2}}\lesssim \|f_r\|_{W_{p_1,q_1}^s} \lesssim_r \|f\|_{W_{p_1,q_1}^s}
\end{equation}
for all $f$ supported in $B(0,2\sqrt{n})$, which implies the embedding $W_{p_1,q_1}^s\subset F_{p_2,q_2}$
by the previous proof for $r=2\sqrt{n}$.
\end{proof}

The following corollary is a direct conclusion of Proposition \ref{proposition, mild characterization} with a slight modification for the cases of $L_1$ and $L_{\infty}$.

\begin{corollary}[Mild characterization of the embedding between local Hardy and Wiener amalgam spaces]\label{corollary, mild characterization}
Let $0<p<\infty$, $0<q\leqslant \infty$, $s\in \mathbb{R}$.
We have
\begin{enumerate}
 \item
  $W_{p,q}^s\subset h_p$
  if and only if one of the following statements holds:
    \begin{enumerate}
     \item
     $\|f\|_{h_p}\lesssim \|f\|_{W_{p,q}^s}$ for any $f\in \mathscr{S}'$ with support in $B(0, r)$,
     \item
     $\|f\|_{h_p}\lesssim \|\hat{f}\|_{L_{q,s}}$ for any $f\in \mathscr{S}'$ with support in $B(0, r)$.
   \end{enumerate}
 \item
  $h_p\subset W_{p,q}^s$
  if and only if one of the following statements holds:
    \begin{enumerate}
     \item
     $\|f\|_{W_{p,q}^s}\lesssim \|f\|_{h_p}$ for any $f\in \mathscr{S}'$ with support in $B(0, r)$,
     \item
     $\|\hat{f}\|_{L_{q,s}}\lesssim\|f\|_{h_p}$ for any $f\in \mathscr{S}'$ with support in $B(0, r)$.
   \end{enumerate}
\end{enumerate}
Moreover, the above statement is true if $p=1$ and $h_1$ is replaced by $L_{1}$,
it is also true if $p=\infty$ and $h_p$ is replaced by $L_{\infty}$.
\end{corollary}
As an application, we prove a key embedding from $W_{p,2}$ into $h_p$.

\begin{proposition}\label{proposition, Littlewood-paley type inequality}
Let $p\in (0, 2]$. Then
\begin{equation}
W_{p,2}\subset h_p.
\end{equation}
\end{proposition}
\begin{proof}
By Corollary \ref{corollary, mild characterization}, we only need to verify
\begin{equation}
  \|f\|_{h_p}\lesssim \|f\|_{L_2}
\end{equation}
for all $f$ with compact support in $B(0, 2\sqrt{n})$. Using the Littlewood-paley characterization of $h_p$, we have
\begin{equation}
  \begin{split}
    \|f\|_{h_p}
    \sim &
    \left\|\left(\sum_{j=0}^{\infty}|\Delta_jf|^2\right)^{1/2}\right\|_{L_p}
    \\
    \lesssim &
    \left\|\left(\sum_{j=0}^{\infty}|\Delta_jf|^2\right)^{1/2}\chi_{B(0, 10\sqrt{n})}\right\|_{L_p}+\left\|\left(\sum_{j=0}^{\infty}|\Delta_jf|^2\right)^{1/2}\chi_{B^c(0, 10\sqrt{n})}\right\|_{L_p}
    \\
    :=&I+II.
  \end{split}
\end{equation}
For I, we use H\"{o}lder's inequality to conclude
  \begin{equation}
  \left\|\left(\sum_{j=0}^{\infty}|\Delta_jf|^2\right)^{1/2}\chi_{B(0, 10\sqrt{n})}\right\|_{L_p}
  \lesssim
  \left\|\left(\sum_{j=0}^{\infty}|\Delta_jf|^2\right)^{1/2}\chi_{B(0, 10\sqrt{n})}\right\|_{L_2}
  \lesssim
  \|f\|_{L_2}.
  \end{equation}
On the other hand,
for $x\in B^c(0, 10\sqrt{n})$, $y\in B(0, 2\sqrt{n})$,  $j\geqslant 1$, we use the fact $|x-y|\geqslant |x|/2$ and the rapid decay of $\psi$ to deduce that
\begin{equation}
  \begin{split}
    \Delta_jf
    =&
    \int_{B(0, 2\sqrt{n})}(\mathscr{F}^{-1}\psi_j)(x-y)f(y)dy
    \\
    =&
    2^{jn}\int_{B(0, 2\sqrt{n})}(\mathscr{F}^{-1}\psi)(2^j(x-y))f(y)dy
    \\
    \lesssim &
    2^{jn}\int_{B(0, 2\sqrt{n})}\langle 2^j(x-y)\rangle^{-\mathscr{L}}f(y)dy
    \\
    \lesssim &
    2^{jn}\langle 2^jx\rangle^{-\mathscr{L}}\int_{B(0, 2\sqrt{n})}f(y)dy
    \lesssim
    2^{j(n-\mathscr{L})}\langle x\rangle^{-\mathscr{L}}\|f\|_{L_2}.
  \end{split}
\end{equation}
We can also verify $\Delta_0f\lesssim \langle x\rangle^{-\mathscr{L}}\|f\|_{L_2}$ for $x\in B^c(0, 10\sqrt{n})$. Thus, we have
\begin{equation}
  \begin{split}
    \left(\sum_{j=0}^{\infty}|\Delta_jf|^2\right)^{1/2}
    \lesssim &
    \left(\sum_{j=0}^{\infty}2^{2j(n-\mathscr{L})}\right)^{1/2}\langle x\rangle^{-\mathscr{L}}\|f\|_{L_2}
    \\
    \lesssim &
    \langle x\rangle^{-\mathscr{L}}\|f\|_{L_2}
  \end{split}
\end{equation}
for $x\in B^c(0, 10\sqrt{n})$.
Thus, we have the following estimate for term II:
\begin{equation}
  \begin{split}
    \left\|\left(\sum_{j=0}^{\infty}|\Delta_jf|^2\right)^{1/2}\chi_{B^c(0, 10\sqrt{n})}\right\|_{L_p}
    \lesssim &
    \|\langle x\rangle^{-\mathscr{L}}\chi_{B^c(0,10\sqrt{n})}\|_{L_p}\cdot\|f\|_{L_2}
    \\
    \lesssim &
    \|f\|_{L_2}.
  \end{split}
\end{equation}
Combining with the estimates of terms I and II, we complete this proof.
\end{proof}
We remark that if $p\in [1,2]$, the inclusion relation $W_{p,2}\subset L_p$  has been concluded in \cite{Cunanan_embedding_Lp_wiener}.
However, the method in \cite{Cunanan_embedding_Lp_wiener} does not apply to $h_p$ for $p\in (0,1]$.

By a further discretization, the embedding relations between Wiener amalgam and Lebesgue spaces can be characterized by inequalities associated with Fourier series.
\begin{proposition}[Characterization of the embedding between Lebesgue and Wiener amalgam spaces by Fourier series]\label{proposition, fourier series}
Let $1\leqslant p\leqslant \infty$, $0<q\leqslant \infty$, $s\in \mathbb{R}$.
We have the following statements:
\begin{enumerate}
 \item
  $W_{p,q}^s\subset L_p$
  if and only if
  \begin{equation}
    \|\sum_{k\in \mathbb{Z}^n}a_ke^{2\pi ikx}\|_{L_p(\mathbb{T}^n)}\lesssim \|\{a_k\}_{k\in \mathbb{Z}^n}\|_{l_q^{s,0}}
  \end{equation}
  for any truncated sequence $\vec{a}=\{a_k\}_{k\in \mathbb{Z}^n}$.
 \item
  $L_p\subset W_{p,q}^s$
  if and only if
  \begin{equation}
    \|\{a_k\}_{k\in \mathbb{Z}^n}\|_{l_q^{s,0}}\lesssim\|\sum_{k\in \mathbb{Z}^n}a_ke^{2\pi ikx}\|_{L_p(\mathbb{T}^n)}
  \end{equation}
  for any truncated sequence $\vec{a}=\{a_k\}_{k\in \mathbb{Z}^n}$.
\end{enumerate}
\end{proposition}
\begin{proof}
We first verify the statement (1).

\textbf{Sufficiency of (1):}
By Corollary \ref{corollary, mild characterization}, to prove $W_{p,q}^s\subset L_p$, we only need to verify
\begin{equation}\label{for proof, characterization, 1}
  \|f\|_{L_p(\mathbb{R}^n)}\lesssim \|\widehat{f}\|_{L_q^{s}(\mathbb{R}^n)}
\end{equation}
for $f$ supported in $\frac{1}{2}Q_0$, where $Q_{k}$ denotes the unit closed cube
centered at $k$.
By a standard limiting argument, we only need to verify (\ref{for proof, characterization, 1}) for $f\in \mathscr{S}$.

Take a smooth cut-off function $\sigma$ with compact supported in $\frac{2}{3}Q_0$, satisfying $\sigma=1$ on $\frac{1}{2}Q_0$.
A fixed function $f\in \mathscr{S}$ supported in $\frac{1}{2}Q_0$ can be reconstructed by
\begin{equation}
f(x)=\sum_{k\in \mathbb{Z}^n}a_ke^{2\pi ik\xi}\sigma(x),
\end{equation}
where the Fourier coefficient $a_k=\widehat{f}(k)$.
Observing that $\widehat{f}(k)$ can be expressed as
\begin{equation}
\widehat{f}(k)=\mathscr{F}(\sigma f)(k)=\int_{\mathbb{R}}\widehat{f}(y)\widehat{\sigma}(k-y)dy.
\end{equation}
For a fixed sufficiently large $N$, we can find a positive smooth function $\phi$ with compact support contained in $\frac{1}{2}Q$, such that
\begin{equation}
\widehat{\phi}(x) \gtrsim \langle x\rangle^{-N}.
\end{equation}
In fact, take $\rho$ to be a smooth function supported in $\frac{1}{2}Q$, such that $\widehat{\rho}$ is nonnegative and $\widehat{\rho}(0)>0$.
The function $\phi$ can be chosen as
\begin{equation}
\phi=\rho\cdot (\langle y\rangle^{-N})^{\vee}.
\end{equation}
Observing
\begin{equation}
|\widehat{\sigma}(k-y)|\lesssim \langle x-y\rangle^{-N}\lesssim  \widehat{\phi}(x-y)
\end{equation}
for $x\in Q_k$,
we obtain that
\begin{equation}
|\widehat{f}(k)|\lesssim \int_{\mathbb{R}}|\widehat{f}(y)||\widehat{\sigma}(k-y)|dy\lesssim \int_{\mathbb{R}}|\widehat{f}(y)||\widehat{\phi}(x-y)|dy=(|\widehat{f}|\ast |\widehat{\phi}|)(x)
\end{equation}
for $x\in Q_k$.
Using Young's inequality or Lemma \ref{lemma, convolution for p<1}, we obtain that
\begin{equation}
\begin{split}
\|\{\widehat{f}(k)\}\|_{l_q^{s,0}}
\lesssim
\||\widehat{f}|\ast |\widehat{\phi}|\|_{L_q^s}
\lesssim
\|\widehat{f}\|_{_{L_q^s}}\|\widehat{\phi}\|_{L_{q\wedge 1}^s}
\lesssim
\|\widehat{f}\|_{L_q^s}.
\end{split}
\end{equation}
Thus, we conclude
\begin{equation}
  \begin{split}
    \|f\|_{L_p(\mathbb{R}^n)}
    = &
    \|\sum_{k\in \mathbb{Z}^n}\widehat{f}(k)e^{2\pi ikx}\|_{L_p(\mathbb{T}^n)}
    \\
    =&
    \lim_{N \rightarrow \infty}\|\sum_{k\in \mathbb{Z}^n}\widehat{f}(k)e^{2\pi ikx}\chi_{\{|k|\leqslant N\}}(k)\|_{L_p(\mathbb{T}^n)}
    \\
    \lesssim &
    \lim_{N \rightarrow \infty}\|\{\widehat{f}(k)\chi_{\{|k|\leqslant N\}}(k)\}_{k\in \mathbb{Z}^n}\|_{l_q^{s,0}}
    \\
    = &
    \|\{\widehat{f}(k)\}_{k\in \mathbb{Z}^n}\|_{l_q^{s,0}}\lesssim \|\widehat{f}\|_{L_q^s(\mathbb{R}^n)}
  \end{split}
\end{equation}
for any $f$ with support contained in $\frac{1}{2}Q_0$.

\medskip

\textbf{Necessity of (1):}
Take a smooth function $\eta$ with compact support in $2Q_0$, satisfying $\eta(x)=1$ for $x\in Q_0$.
For a fixed truncated sequence $\vec{a}=\{a_k\}_{k\in \mathbb{Z}^n}$,
we define
\begin{equation}
  f(x)=\eta(x)\sum_{k\in \mathbb{Z}^n}a_ke^{2\pi ikx}.
\end{equation}
Obviously, $f$ is a smooth function supported in $2Q_0$.
Observing that
\begin{equation}
  \widehat{f}(\xi)=\sum_{k\in \mathbb{Z}^n}a_k\widehat{\eta}(\xi-k),
\end{equation}
and using the rapid decay of $\widehat{\eta}$,
we obtain
\begin{equation}
  |\widehat{f}(\xi)|
  \lesssim
  \sum_{k\in \mathbb{Z}^n}|a_k|\langle l-k\rangle^{-\mathscr{L}}
\end{equation}
for $\xi\in Q_l$.
Denoting $|\vec{a}|=\{|a_k|\}_{k\in \mathbb{Z}^n}$,
we deduce that
\begin{equation}
\begin{split}
  \|\widehat{f}\|_{L_q^s}
  \lesssim &
  \|\{(|\vec{a}|\ast \langle \cdot\rangle^{-\mathscr{L}})(k)\}_{k\in \mathbb{Z}^n}\|_{l_q^{s,0}}
  \\
  \lesssim &
  \|\{\langle k\rangle^{-\mathscr{L}}\}_{k\in \mathbb{Z}^n}\|_{l_{q\wedge 1}^{|s|,0}} \|\vec{a}\|_{l_q^{s,0}}
  \lesssim
  \|\vec{a}\|_{l_q^{s,0}},
\end{split}
\end{equation}
where the last inequality is an application of Lemma \ref{lemma, convolution for lp}.
Thus, we duduce
\begin{equation}
  \begin{split}
    \|\sum_{k\in \mathbb{Z}^n}a_ke^{2\pi ikx}\|_{L_p(\mathbb{T}^n)}
    \lesssim &
    \|\sum_{k\in \mathbb{Z}^n}\eta(x)a_ke^{2\pi ikx}\|_{L_p(\mathbb{R}^n)}
    \\
    = &
    \|f\|_{L_p(\mathbb{R}^n)}
    \lesssim
    \|\widehat{f}\|_{L_q^s(\mathbb{R}^n)}
    \lesssim
    \|\{a_k\}_{k\in \mathbb{Z}^n}\|_{l_q^{s,0}}.
  \end{split}
\end{equation}

Now, we turn to the proof of statement (2).

\medskip

\textbf{Sufficiency of (2):}
Using Corollary \ref{corollary, mild characterization} and a limiting argument,
we only need to verify
\begin{equation}
  \|\widehat{f}\|_{L_q^s(\mathbb{R}^n)}\lesssim \|f\|_{L_p(\mathbb{R}^n)}
\end{equation}
for any fixed Schwartz function $f$ with support in $\frac{1}{2}Q_0$.
For a fixed smooth cut-off function $\sigma$ with compact supported in $\frac{2}{3}Q_0$, satisfying $\sigma=1$ on $\frac{1}{2}Q_0$,
a fixed $f\in \mathscr{S}$ supported in $\frac{1}{2}Q_0$ can be reconstructed by
\begin{equation}
f(x)=\sum_{k\in \mathbb{Z}^n}a_ke^{2\pi ik\xi}\sigma(x),
\end{equation}
where the Fourier coefficient $a_k=\widehat{f}(k)$.
By a similar argument as the proof in Necessity part of (1), we deduce
\begin{equation}
  \begin{split}
    \|\widehat{f}\|_{L_q^s(\mathbb{R}^n)}
    \lesssim &
    \|\{\widehat{f}(k)\}_{k\in \mathbb{Z}^n}\|_{l_q^{s,0}}
    \\
    =&
    \lim_{N \rightarrow \infty}\|\{\widehat{f}(k)\chi_{\{|k|\leqslant N\}}(k)\}_{k\in \mathbb{Z}^n}\|_{l_q^{s,0}}
    \\
    \lesssim &
    \lim_{N \rightarrow \infty}\|\sum_{k\in \mathbb{Z}^n}\widehat{f}(k)e^{2\pi ikx}\chi_{\{|k|\leqslant N\}}(k)\|_{L_p(\mathbb{T}^n)}
    \\
    = &
    \|\sum_{k\in \mathbb{Z}^n}\widehat{f}(k)e^{2\pi ikx}\|_{L_p(\mathbb{T}^n)}
    =
    \|f\|_{L_p(\mathbb{R}^n)},
  \end{split}
\end{equation}
which is the desired conclusion.

\medskip

\textbf{Necessity of (2):}
For a truncated sequence $\{a_k\}_{k\in \mathbb{Z}^n}$, we define
\begin{equation}
  f=\chi_{Q_0}\sum_{k\in \mathbb{Z}^n}a_ke^{2\pi ikx}.
\end{equation}
Choose a smooth function $\eta$ supported in $2Q_0$ such that $\eta(x)=1$ in $Q_0$.
We have $f=f\eta$, and then $\widehat{f}=\widehat{f}\ast \widehat{\eta}$.

Using the same argument as in the proof of Sufficiency part of (1). We have
\begin{equation}
\begin{split}
\|\{a_k\}_{k\in \mathbb{Z}^n}\|_{l_q^{s,0}}
=
\|\{\widehat{f}(k)\}\|_{l_q^{s,0}}
\lesssim
\|\widehat{f}\|_{L_q^s}.
\end{split}
\end{equation}
Thus,
\begin{equation}
  \|\{a_k\}_{k\in \mathbb{Z}^n}\|_{l_q^{s,0}}
  \lesssim
  \|\widehat{f}\|_{L_q^s(\mathbb{R}^n)}
  \lesssim
  \|f\|_{L_p(\mathbb{R}^n)}
  =
  \|\sum_{k\in \mathbb{Z}^n}a_ke^{2\pi ikx}\|_{L_p(\mathbb{T}^n)},
\end{equation}
which is the desired conclusion.
\end{proof}

\begin{remark}
By Corollary \ref{corollary, mild characterization} and Proposition \ref{proposition, fourier series}, we actually establish some relations between
old inequalities and new inequalities.
For instance, let $p,q\geqslant 2$, Theorem \ref{theorem, embedding between Wiener and hp} and Proposition \ref{proposition, fourier series} imply that
  \begin{equation}
    \|\sum_{k\in \mathbb{Z}^n}a_ke^{2\pi ikx}\|_{L_p(\mathbb{T}^n)}\lesssim \|\{a_k\}_{k\in \mathbb{Z}^n}\|_{l_p^{n(1-2/p),0}},
  \end{equation}
which is just the Theorem 6 in \cite{Hardy}.
\end{remark}

\section{Inclusion relations between $W_{p,q}^s$ and $B_{p,q}$}

This section is devoted to the proof of Theorem \ref{theorem, embedding between Wiener and Besov}.
For the sufficiency part, we actually only need to deal with some endpoint cases, then the final conclusion follows by an interpolation argument.
For the necessity, we show that the embedding relations between Besov and Wiener amalgam spaces actually imply
embedding relations about corresponding weighted sequences. Then, Lemma \ref{lemma, Sharpness of embedding, for dyadic decomposition} ensures the necessity.

\subsection{Sufficiency}\

\medskip

\textbf{Sufficiency for $W_{p,q}^s\subset B_{p,q}$}

\medskip

\textbf{(i). For $1/p\geqslant 1/q$:}
We firstly verify
\begin{equation}\label{for proof, theorem Besov, 4}
  W_{p,\infty}^{n/2}\subset B_{p,\infty}
\end{equation}
for $p\leqslant 2$.
Take the window function $\phi$ to be a smooth function with Fourier support on $B(0,1)$.
Obviously, we have $\langle \xi\rangle \sim 2^j$ if $V_{\phi}(\Delta_jf)(x,\xi)\neq 0$.
We also have $|\{\xi: V_{\phi}(\Delta_jf)(x,\xi)\neq 0\}|\sim 2^{jn}$.
Using Proposition \ref{proposition, Littlewood-paley type inequality} and H\"{o}lder's inequality, we deduce
\begin{equation}
  \begin{split}
    \|\Delta_j f\|_{L_p}
    \lesssim
    \|\Delta_j f\|_{h_p}
    \lesssim &
    \|\Delta_j f\|_{W_{p,2}}
    \\
    = &
    \left\|\|V_{\phi}(\Delta_jf)(x,\xi)\|_{L_{\xi,2}}\right\|_{L_{x,p}}
    \\
    \lesssim &
    \left\|2^{jn/2}\|V_{\phi}(\Delta_jf)(x,\xi)\|_{L_{\infty}}\right\|_{L_p}
    \\
    \sim &
    \left\|\|V_{\phi}(\Delta_jf)(x,\xi)\|_{L_{\infty}^{n/2}}\right\|_{L_p}
    \sim
    \|\Delta_jf\|_{W_{p,\infty}^{n/2}}.
  \end{split}
\end{equation}
Using the convolution inequality $W_{p\wedge 1,\infty}\ast W_{p,\infty}^s\subset W_{p,\infty}^s$ (see \cite{Guo_Characterization}) , we deduce
\begin{equation}
  \begin{split}
    \|\Delta_jf\|_{W_{p,\infty}^{n/2}}
    = &
    \|\check{\psi_j} \ast f\|_{W_{p,\infty}^{n/2}}
    \\
    \lesssim &
    \|\check{\psi_j}\|_{W_{p\wedge 1,\infty}}\cdot \|f\|_{W_{p,\infty}^{n/2}}.
  \end{split}
\end{equation}
Moreover, for a fixed $\xi$ such that  $V_{\phi}\check{\psi_j}(x,\xi) \neq 0$ for some $j$, we have
\begin{equation}
  \begin{split}
    |V_{\phi}\check{\psi_j}(x,\xi)|
    =&
    \left|\int_{\mathbb{R}^n}\check{\psi_j}(y)\bar{\phi}(y-x)e^{-2\pi iy\cdot \xi}dy\right|
    \\
    \lesssim &
    \left|\int_{B(0,\langle x\rangle/2)}\right|+\left|\int_{B(0,\langle x\rangle/2)^c}\right|: = A+B.
  \end{split}
\end{equation}
Using the rapid decay of $\phi$, we have
\begin{equation}
  \begin{split}
    A
    =&
    \left|\int_{B(0,\langle x\rangle/2)}\check{\psi_j}(y)\bar{\phi}(y-x)e^{-2\pi iy\cdot \xi}dy\right|
    \\
    \lesssim &
    \langle x\rangle^{-\mathscr{L}}\int_{B(0,\langle x\rangle/2)}|\check{\psi_j}(y)|dy
    \\
    \lesssim &
    \langle x\rangle^{-\mathscr{L}}\|\check{\psi_j}\|_{L_1}\lesssim \langle x\rangle^{-\mathscr{L}}.
  \end{split}
\end{equation}
Using the rapid decay of $h$, we have
\begin{equation}
  \begin{split}
    B
    =&
    \left|\int_{B(0,\langle x\rangle/2)^c}\check{\psi_j}(y)\bar{\phi}(y-x)e^{-2\pi iy\cdot \xi}dy\right|
    \\
    \lesssim &
    \left|2^{jn}\int_{B(0,\langle x\rangle/2)^c}\check{\psi}(2^{j}y)\bar{\phi}(y-x)e^{-2\pi iy\cdot \xi}dy\right|
    \\
    \lesssim &
    2^{jn}\langle 2^j\langle x\rangle \rangle^{-\mathscr{L}}\|\phi\|_{L_1}\lesssim \langle x\rangle^{-\mathscr{L}}
  \end{split}
\end{equation}
for $j\geqslant 1$. With a slight modification, we can also conclude
\begin{equation}
  \left|\int_{B(0,\langle x\rangle/2)^c}\check{\psi_0}(y)\bar{\phi}(y-x)e^{-2\pi iy\cdot \xi}dy\right|\lesssim \langle x\rangle^{-\mathscr{L}}.
\end{equation}
Combing with the estimates of term I and II, we have $|V_{\phi}\check{\psi_j}(x,\xi)|\lesssim \langle x\rangle^{-\mathscr{L}}$ uniformly for all $j\in \mathbb{N}$.
By the definition of $W_{p,\infty}$,
\begin{equation}
  \|\check{\psi_j}\|_{W_{p\wedge 1,\infty}}\lesssim \|\langle x\rangle^{-\mathscr{L}}\|_{L_{p\wedge 1}}\lesssim 1,
\end{equation}
which implies that
\begin{equation}
   \|\Delta_j f\|_{L_p}\lesssim \|\Delta_jf\|_{W_{p,\infty}^{n/2}}\lesssim \|\check{\psi_j}\|_{W_{p\wedge 1,\infty}}\cdot \|f\|_{W_{p,\infty}^{n/2}} \lesssim \|f\|_{W_{p,\infty}^{n/2}}
\end{equation}
for all $j\in \mathbb{N}$. Taking $l^{\infty}$-norm on the left side, we obtain
\begin{equation}
  \|f\|_{B_{p,\infty}}\lesssim \|f\|_{W_{p,\infty}^{n/2}}
\end{equation}
for $p\leqslant 2$.
On the other hand, we use Proposition \ref{proposition, Littlewood-paley type inequality} to deduce
\begin{equation}\label{for proof, theorem Besov, 5}
  W_{p,2}\subset h_p\sim F_{p,2}\subset B_{p,2}
\end{equation}
for $p\leqslant 2$.  We also have
\begin{equation}\label{for proof, theorem Besov, 6}
  W_{p,p}^{\alpha(p,p)}=M_{p,p}^{\alpha(p,p)}\subset B_{p,p}
\end{equation}
for $0<p\leqslant \infty$
by the known inclusion between modulation and Besov space (see \cite{Guo_Zhao_interpolation,Wang_Han,Sugimoto_Tomita}).
Now, the desired conclusion follows by an interpolation argument among
(\ref{for proof, theorem Besov, 4}), (\ref{for proof, theorem Besov, 5}) and (\ref{for proof, theorem Besov, 6}).

\medskip

\textbf{(ii). For $1/p< 1/q$:}
If $2\leqslant p\leqslant \infty$,
we use the Haudsorff-Young inequality $\|f\|_{L_p}\lesssim \|\hat{f}\|_{L_{p'}}$ and Corollary \ref{corollary, mild characterization} to deduce
$W_{p,p'}\subset L_p$. Thus, we have
\begin{equation}
  W_{p,q}^{\alpha(p,q)+\epsilon}\subset W_{p,p'}\subset L_p\subset B_{p,q}^{-\epsilon}
\end{equation}
for any $\epsilon>0$. By the potential lifting, we actually conclude
$W_{p,q}^s\subset B_{p,q}$
for $s>\alpha(p,q)$.

If $p\leqslant 2$, we recall $ W_{p,p}=M_{p,p}\subset B_{p,p}$ for $0<p\leqslant 2$ (see \cite{Guo_Zhao_interpolation,Wang_Han,Sugimoto_Tomita}). Then
\begin{equation}
  W_{p,q}^{\alpha(p,q)+\epsilon}=W_{p,q}^{\epsilon}\subset W_{p,p}^{\epsilon} \subset B_{p,p}^{\epsilon} \subset B_{p,q}
\end{equation}
for any $\epsilon >0$.

\bigskip

\textbf{Sufficiency for $B_{p,q}\subset W_{p,q}^s$}

\medskip

\textbf{(i). For $1/p\leqslant 1/q$:}
We firstly show
\begin{equation}\label{for proof, theorem Besov, 8}
  B_{p,q}\subset W_{p,q}^{n(1/2-1/q)}
\end{equation}
for $2\leqslant p\leqslant \infty$, $q\leqslant 1$. We have
\begin{equation}\label{for proof, theorem Besov, 7}
  \begin{split}
    \|f\|_{W_{p,q}^{n(1/2-1/q)}}
    \lesssim
    \left(\sum_{j\in \mathbb{N}}\|\Delta_jf\|^q_{W_{p,q}^{n(1/2-1/q)}}\right)^{1/q}.
  \end{split}
\end{equation}
Observing that $|\{\xi: V_{\phi}(\Delta_jf)(x,\xi)\neq 0\}|\sim 2^{jn}$ and $\langle\xi\rangle \sim 2^j$ if $V_{\phi}(\Delta_jf)(x,\xi)\neq 0$.
We use H\"{o}lder's inequality to deduce that
\begin{equation}
  \begin{split}
    \|\Delta_jf\|_{W_{p,q}^{n(1/2-1/q)}}
    = &
    \left\|\|V_{\phi}(\Delta_jf)(x,\xi)\|_{L_q^{n(1/2-1/q)}}\right\|_{L_p}
    \\
    \sim &
    2^{jn(1/2-1/q)}\left\|\|V_{\phi}(\Delta_jf)(x,\xi)\|_{L_q}\right\|_{L_p}
    \\
    \lesssim &
    2^{jn(1/2-1/q)}\cdot 2^{jn(1/q-1/2)}\left\|\|V_{\phi}(\Delta_jf)(x,\xi)\|_{L_2}\right\|_{L_p}
    \\
    = &
    \|\Delta_jf\|_{W_{p,2}}.
  \end{split}
\end{equation}
Using $L_p\subset W_{p,2}$ for $2\leqslant p\leqslant \infty$ proved in Theorem B, we deduce
\begin{equation}
  \|\Delta_jf\|_{W_{p,q}^{n(1/2-1/q)}}\lesssim \|\Delta_jf\|_{W_{p,2}}\lesssim \|\Delta_jf\|_{L_p}.
\end{equation}
Thus, we use (\ref{for proof, theorem Besov, 7}) to deduce
\begin{equation*}
  \|f\|_{W_{p,q}^{n(1/2-1/q)}}
  \lesssim
  \left(\sum_{j\in \mathbb{N}}\|\Delta_jf\|^q_{W_{p,q}^{n(1/2-1/q)}}\right)^{1/q}
  \lesssim
  \left(\sum_{j\in \mathbb{N}}\|\Delta_jf\|^q_{L_p}\right)^{1/q}\sim \|f\|_{B_{p,q}}.
\end{equation*}
On the other hand, by the duality of (\ref{for proof, theorem Besov, 5}), we obtain
\begin{equation}\label{for proof, theorem Besov, 9}
  B_{p,2}\subset W_{p,2}
\end{equation}
for $2\leqslant p\leqslant \infty$. We also recall
\begin{equation}\label{for proof, theorem Besov, 10}
  B_{p,p}\subset M_{p,p}^{\beta(p,p)}=W_{p,p}^{\beta(p,p)}
\end{equation}
by the known inclusion between modulation and Besov space (see \cite{Guo_Zhao_interpolation,Wang_Han,Sugimoto_Tomita}).
By an interpolation argument among
(\ref{for proof, theorem Besov, 8}), (\ref{for proof, theorem Besov, 9}) and (\ref{for proof, theorem Besov, 10}),
we obtain the desired conclusion.

\medskip

\textbf{(ii). For $1/p> 1/q$:}
By the inclusion $h_p\subset W_{p,q}^{\beta(p,q)}$ to be proved in next section, we obtain
\begin{equation}
  B_{p,q}\subset F_{p,2}^{-\varepsilon} \subset W_{p,q}^{\beta(p,q)-\varepsilon}
\end{equation}
for any $\varepsilon >0$.

\subsection{Necessity}
\begin{proposition}\label{proposition, necessity for Besov, 1}
Let $0<p, q\leqslant \infty$, $s\in \mathbb{R}$. Then we have
\begin{enumerate}
  \item
  $W_{p,q}^s\subset B_{p,q} \Longrightarrow l_p^{s+n/q,1}\subset l_q^{n(1-1/p),1}$,
  \item
  $B_{p,q}\subset W_{p,q}^s \Longrightarrow l_q^{n(1-1/p),1}\subset l_p^{s+n/q,1}$.
\end{enumerate}
\end{proposition}
\begin{proof}
We only state the proof for $p\neq \infty$. The case $p=\infty$ can be deduced similarly with a slight modification.
In this proof, we take the window function $\phi$ to be a smooth function with small Fourier support in $B(0,10^{-10})$.
Choose a smooth function $h$ with Fourier support on $3/4\leqslant |\xi|\leqslant 4/3$,
satisfying $\hat{h}(\xi)=1$ on $7/8\leqslant |\xi|\leqslant 8/7$.
By the assumption of $h$, we have
$\Delta_jh_j=h_j$ for all $j\in \mathbb{N}$, where $\widehat{h_j}(\xi)=\hat{h}(\xi/2^j)$.
Denote
$$F_N=\sum_{j=0}^{\infty}a_jT_{Nje_0}h_j,$$ where $\{a_j\}_{j=0}^{\infty}$ is a truncated (only finite nonzero items) sequence of nonnegative real number,
$e_0=(1,0,\cdots,0)$ is the unit vector of $\mathbb{R}^n$, $N$ is a sufficient large number to be chosen later.

We first estimate the norm of $B_{p,q}$. A direct calculation yields
\begin{equation}
  \begin{split}
    \|F_N\|_{B_{p,q}}
    = &
    \|\{\|\Delta_j F_N\|_{L_p}\}_{j\in \mathbb{N}} \|_{l_q^{0,1}}
    \\
    = &
    \|\{\|a_j\Delta_j(T_{Nje_0}h_j)\|_{L_p}\}_{j\in \mathbb{N}} \|_{l_q^{0,1}}
    \\
    = &
    \|\{a_j\|h_j\|_{L_p}\}_{j\in \mathbb{N}}\|_{l_q^{0,1}}
    \\
    \sim &
    \|\{2^{jn(1-1/p)}a_j\}_{j\in \mathbb{N}}\|_{l_q^{0,1}}\sim \|\{a_j\}\|_{l_q^{n(1-1/p),1}}.
  \end{split}
\end{equation}

Next, we estimate the norm of $W_{p,q}^s$.
By the almost orthogonality of $\{a_jT_{Nje_0}h_j\}_{j=0}^{\infty}$ as $N\rightarrow \infty$, we deduce
\begin{equation}
  \begin{split}
      \lim_{N\rightarrow \infty}\|F_N\|_{W_{p,q}^s}
      = &
      \lim_{N\rightarrow \infty}\|\sum_{j=0}^{\infty}a_jT_{Nje_0}h_j\|_{W_{p,q}^s}
      \\
      = &
      \lim_{N\rightarrow \infty}\left(\sum_{j=0}^{\infty}\|a_jT_{Nje_0}h_j\|^p_{W_{p,q}^s}\right)^{1/p}
      \\
      = &
      \left(\sum_{j=0}^{\infty}\|a_jh_j\|^p_{W_{p,q}^s}\right)^{1/p}.
  \end{split}
\end{equation}
By the same argument as in subsection 4.1, we have
\begin{equation}
  |V_{\phi}h_j(x,\xi)|\lesssim \langle x\rangle^{-\mathscr{L}}
\end{equation}
for all $j\in \mathbb{N}$.
Moreover, by the definition of $h$ and $\phi$, we have $\langle \xi\rangle \sim 2^j$ when $V_{\phi}h_j(x,\xi) \neq 0$.
We also have $|\{\xi: V_{\phi}(h_j)(x,\xi)\neq 0\}|\sim 2^{jn}$.
Thus, we obtain
\begin{equation}
  \|V_{\phi}h_j(x,\xi)\|_{L_q^s}
  \lesssim
  2^{j(s+n/q)}\langle x\rangle^{-\mathscr{L}},
\end{equation}
and
\begin{equation}
  \begin{split}
    \|h_j\|_{W_{p,q}^s}
    = &
    \left\|\|V_{\phi}h_j(x,\xi)\|_{L_q^s}\right\|_{L_p}
    \\
    \lesssim &
    2^{j(s+n/q)}\|\langle x\rangle^{-\mathscr{L}}\|_{L_p}\lesssim 2^{j(s+n/q)}.
  \end{split}
\end{equation}

On the other hand, we denote $E_j=\{\xi: \widehat{h_j}(\cdot)\overline{\widehat{\phi}}(\cdot-\xi)=\overline{\widehat{\phi}}(\cdot-\xi)\}$.
By the definition of $\phi$ and $h_j$,
we have $|V_{\phi}h_j(x,\xi)|=|\phi(-x)|$ and $\langle\xi \rangle \sim 2^j$ for $\xi \in E_j$.
We also have $|E_j|\sim 2^{jn}$.
Thus, we estimate the lower bound by
\begin{equation}
  \begin{split}
    \|h_j\|_{W_{p,q}^s}
    = &
    \left\|\|V_{\phi}h_j(x,\xi)\|_{L_{\xi,q}^s}\right\|_{L_{x,p}}
    \\
    \gtrsim  &
    \left\|\|V_{\phi}h_j(x,\xi)\chi_{E_j}(\xi)\|_{L_{\xi,q}^s}\right\|_{L_{x,p}}
    \\
    = &
    \left\|\|\phi(-x)\chi_{E_j}(\xi)\|_{L_{\xi,q}^s}\right\|_{L_{x,p}}
    =
    \|\chi_{E_j}\|_{L_q^s}\cdot\|\phi\|_{L_p}
    \sim
    2^{j(s+n/q)}.
  \end{split}
\end{equation}
Thus, we have
\begin{equation}
  \|h_j\|_{W_{p,q}^s}\sim 2^{j(s+n/q)}
\end{equation}
for all $j\in \mathbb{N}$,
which implies
\begin{equation}
  \begin{split}
      \lim_{N\rightarrow \infty}\|F_N\|_{W_{p,q}^s}
      = &
      \left(\sum_{j=0}^{\infty}\|a_jh_j\|^p_{W_{p,q}^s}\right)^{1/p}
      \sim
      \|\{a_j\}\|_{l_p^{s+n/q,1}}.
  \end{split}
\end{equation}

If $W_{p,q}^s\subset B_{p,q}$, we have $\|F_N\|_{B_{p,q}}\lesssim \|F_N\|_{W_{p,q}^s}$.
Letting $N\rightarrow \infty$ on both sides of the above inequality, we deduce
\begin{equation}
  \|\{a_j\}\|_{l_q^{n(1-1/p),1}}\lesssim \|\{a_j\}\|_{l_p^{s+n/q,1}}
\end{equation}
for any truncated sequence of nonnegative real number $\{a_j\}_{j=0}^{\infty}$, which implies the desired embedding $l_p^{s+n/q,1}\subset l_q^{n(1-1/p),1}$.
Similarly, if $B_{p,q}\subset W_{p,q}^s$, we obtain the embedding $l_q^{n(1-1/p),1}\subset l_p^{s+n/q,1}$.
\end{proof}

\begin{proposition}\label{proposition, necessity for Besov, 2}
Let $0<p,q\leqslant \infty$, $s\in \mathbb{R}$. Then we have
\begin{enumerate}
  \item
  $W_{p,q}^s\subset B_{p,q} \Longrightarrow l_p^{s+n/p,1}\subset l_q^{n/p,1}$,
  \item
  $B_{p,q}\subset W_{p,q}^s \Longrightarrow l_q^{n/p,1}\subset l_p^{s+n/p,1}$.
\end{enumerate}
\end{proposition}
\begin{proof}\
We only state the proof for $p\neq \infty$, the case $p=\infty$ can be handled similarly.
Take $g$ to be a smooth function with Fourier support in $B(0, 10^{-10})$.
Denote $\hat{g_k}(\xi)=\hat{g}(\xi-k)$ and
\begin{equation}
\Gamma_j=\{k\in \mathbb{Z}^n:~\Delta_jg_k=g_k \}.
\end{equation}
We have $\langle k\rangle \sim 2^j$ for $k\in \Gamma_j$,
and we also have $|\Gamma_j|\sim 2^{jn}$.
By Lemma\ref{lemma, embedding between modulation and wiener}, we have
\begin{equation}
  \|g_k\|_{M_{p,\max\{p,q\}}^s}\lesssim \|g_k\|_{W_{p,q}^s}\lesssim \|g_k\|_{M_{p,\min\{p,q\}}^s}.
\end{equation}
By the discrete definition of modulation space (see 6.2 in \cite{Wang_book}), we can deduce
$\|g_k\|_{M_{p,\max\{p,q\}}^s}\sim \langle k\rangle^{s}\|g_k\|_{L_p}\sim \|g_k\|_{M_{p,\min\{p,q\}}^s}$,
which yields
\begin{equation}
  \|g_k\|_{W_{p,q}^s}\sim \langle k\rangle^{s}\|g_k\|_{L_p}.
\end{equation}
Let $\{b_j\}_{j=0}^{\infty}$ be a truncated (only finite nonzero items) sequence of nonnegative real number,
$N$ is a sufficient large number to be chosen later.
Let $$G_N=\sum_{j=0}^{\infty}b_j\sum_{k\in \Gamma_j}T_{Nk}g_k.$$
By the almost orthogonality of $\{T_{Nk}g_k\}_{k\in \mathbb{Z}^n}$ as $N\rightarrow \infty$, we deduce
\begin{equation}
  \begin{split}
    \lim_{N\rightarrow \infty}\|G_N\|_{W_{p,q}^s}
    =&
    \lim_{N\rightarrow \infty}\|\sum_{j=0}^{\infty}b_j^p\sum_{k\in \Gamma_j}T_{Nk}g_k\|_{W_{p,q}^s}
    \\
    =&
    \lim_{N\rightarrow \infty}\left(\sum_{j=0}^{\infty}b_j^p\sum_{k\in \Gamma_j}\|T_{Nk}g_k\|^p_{W_{p,q}^s}\right)^{1/p}
    \\
    =&
    \lim_{N\rightarrow \infty}\left(\sum_{j=0}^{\infty}b_j^p\sum_{k\in \Gamma_j}2^{jsp}\|g_k\|^p_{L_p}\right)^{1/p}
    \sim
    \|\{b_j\}\|_{l_p^{s+n/p}}.
  \end{split}
\end{equation}
On the other hand, for every $j\in \mathbb{N}$, we have
\begin{equation}
  \begin{split}
    \lim_{N\rightarrow \infty}\|\Delta_jG_N\|_{L_p}
    =&
    \lim_{N\rightarrow \infty}\|\sum_{k\in \Gamma_j}b_jT_{Nk}g_k\|_{L_p}
    \\
    =&
    \lim_{N\rightarrow \infty}\left(\sum_{k\in \Gamma_j}\|b_jT_{Nk}g_k\|^p_{L_p}\right)^{1/p}
    \\
    =&
    \left(\sum_{k\in \Gamma_j}b_j^p\|g_k\|^p_{L_p}\right)^{1/p}
    \sim
    2^{jn/p}b_j.
  \end{split}
\end{equation}
Thus,
\begin{equation}
  \begin{split}
    \lim_{N\rightarrow \infty}\|G_N\|_{B_{p,q}}
    =&
    \lim_{N\rightarrow \infty}\|\{\|\Delta_jG_N\|_{L_p}\}_{j\in \mathbb{N}}\|_{l_q^{0,1}}
    \\
    \sim &
    \|\{2^{jn/p}b_j\}_{j\in \mathbb{N}}\|_{l_q^{0,1}}
    \sim
    \|\{b_j\}_{j\in \mathbb{N}}\|_{l_q^{n/p,1}}.
  \end{split}
\end{equation}
Combining with the above estimates, we use $W_{p,q}^s\subset B_{p,q}$ to deduce
\begin{equation}
  \|\{b_j\}\|_{l_q^{n/p,1}}\lesssim \|\{b_j\}\|_{l_p^{s+n/p}},
\end{equation}
which implies the desired embedding $l_p^{s+n/p}\subset l_q^{n/p,1}$.
Similarly, if $B_{p,q}\subset W_{p,q}^s$, we deduce the embedding $l_q^{n/p,1}\subset l_p^{s+n/p}$.
\end{proof}

Now, we can give the proof for the necessity of Theorem \ref{theorem, embedding between Wiener and Besov},
we will use Proposition \ref{proposition, forbidden for wiener by probability}, which will be independently verified in Section 5.

\medskip

\textbf{Necessity for $W_{p,q}^s\subset B_{p,q}$:} By Propositions \ref{proposition, necessity for Besov, 1} and \ref{proposition, necessity for Besov, 2}, we
deduce
$$l_p^{s+n/q,1}\subset l_q^{n(1-1/p),1},\hspace{6mm}l_p^{s+n/p,1}\subset l_q^{n/p,1}.$$
By Lemma \ref{lemma, Sharpness of embedding, for dyadic decomposition},
$l_p^{s+n/q,1}\subset l_q^{n(1-1/p),1}$ implies $s\geqslant n(1-1/p-1/q)$, where the inequality is strict if $1/q>1/p$.
Also, $l_p^{s+n/p,1}\subset l_q^{n/p,1}$ implies $s\geqslant 0$ and the inequality is strict if $1/q>1/p$.

In addition, $W_{p,q}^s\subset B_{p,q}$ implies $W_{p,q}^{s+\epsilon}\subset h_p$ for $0<p< \infty$, where $\epsilon$ is any fixed positive number.
Using Proposition \ref{proposition, forbidden for wiener by probability}, $W_{p,q}^{s+\epsilon}\subset h_p$ implies $l_{q}^{s+\epsilon,0}\subset l_2^{0,0}$
for $0<p\leqslant 2\leqslant q\leqslant \infty$.
Using Lemma \ref{lemma, Sharpness of embedding, for uniform decomposition}, $l_{q}^{s+\epsilon,0}\subset l_2^{0,0}$ implies $s+\epsilon \geqslant n(1/2-1/q)$.
Letting $\epsilon\rightarrow 0$, we deduce that $s \geqslant n(1/2-1/q)$.
By the above arguments, we actually verify that $W_{p,q}^s\subset B_{p,q}$ implies $s \geqslant n(1/2-1/q)$ for $0<p\leqslant 2\leqslant q\leqslant \infty$.

Combing with the above estimates, we obtain $s\geqslant \alpha(p,q)$ for $1/q\leqslant 1/p$, and $s> \alpha(p,q)$ for $1/q>1/p$, which is the desired conclusion.

\medskip

\textbf{Necessity for $B_{p,q}\subset W_{p,q}^s$:} By Propositions \ref{proposition, necessity for Besov, 1} and \ref{proposition, necessity for Besov, 2}, we
deduce
$$l_q^{n(1-1/p),1}\subset l_p^{s+n/q,1},\hspace{6mm}l_q^{n/p,1}\subset l_p^{s+n/p,1}.$$
By Lemma \ref{lemma, Sharpness of embedding, for dyadic decomposition},
$l_q^{n(1-1/p),1}\subset l_p^{s+n/q,1}$ implies $s\leqslant n(1-1/p-1/q)$, where the inequality is strict if $1/p>1/q$.
Also, $l_q^{n/p,1}\subset l_p^{s+n/p,1}$ implies $s\leqslant 0$ and the inequality is strict if $1/p>1/q$.

In addition, $B_{p,q}\subset W_{p,q}^s$ implies $h_p \subset W_{p,q}^{s-\epsilon}$ for $2\leqslant p< \infty$, where $\epsilon$ is any fixed positive number.
Using Proposition \ref{proposition, forbidden for wiener by probability}, $h_p \subset W_{p,q}^{s-\epsilon}$ implies $l_2^{0,0}\subset l_{q}^{s-\epsilon,0}$
for $1<p< \infty$.
Thus, $s-\epsilon \leqslant n(1/2-1/q)$ follows by Lemma \ref{lemma, Sharpness of embedding, for uniform decomposition}.
Letting $\epsilon\rightarrow 0$, we deduce that $s \leqslant n(1/2-1/q)$.
By the above arguments, we actually verify that $B_{p,q}\subset W_{p,q}^s$ implies $s \leqslant n(1/2-1/q)$ for $2\leqslant p<\infty$.
By an interpolation with $B_{2,q}\subset W_{2,q}^{n(1/2-1/q)}$, we can use a contradiction argument to show that $B_{\infty,q}\subset W_{\infty,q}^s$ implies $s \leqslant n(1/2-1/q)$.

Combing with the above estimates, we obtain $s\leqslant \alpha(p,q)$ for $1/p\leqslant 1/q$, and $s< \alpha(p,q)$ for $1/p>1/q$, which is the desired conclusion.

\section{Inclusion relations between $W_{p,q}^s$ and $h_p$}

In this section, we will give the proof for Theorem \ref{theorem, embedding between Wiener and hp}.
By the feature of dyadic decomposition, if we ignore all the critical cases,
Theorem \ref{theorem, embedding between Wiener and Besov} and Theorem \ref{theorem, embedding between Wiener and hp} has no essential difference.
Nevertheless, these two theorems are quite different in the critical cases $s=\alpha(p,q)$ and $s=\beta(p,q)$.
For instance, for $0<p<\infty,\, p< q$, Theorem \ref{theorem, embedding between Wiener and Besov} says that the embedding relation
\begin{equation}
  B_{p,q}\subset W_{p,q}^{\beta(p,q)}
\end{equation}
is negative. However, $h_p\subset W_{p,q}^{\beta(p,q)}$ is an immediate conclusion of Theorem \ref{theorem, embedding between Wiener and hp}.

\subsection{Sufficiency}

\begin{proposition}\label{proposition, atom estimate}
Let $0<p<\infty$, $0<q\leqslant \infty$ satisfy one of the following conditions
\begin{equation}\label{for proof, theorem Lp, 4}
  (i).~ p<1, ~0<1/q\leqslant 1/p;\quad~(ii).~ p\leqslant 1, ~q=\infty.
\end{equation}
Then, we have
\begin{equation}
\|a\|_{W_{p,q}^{n(1-1/p-1/q)}}\lesssim 1
\end{equation}
for all $h_p$-atoms $a$.
\end{proposition}
\begin{proof}
We first take $a$ to be an $h_p$ atom, and without loss of generality we may assume tha  $a$ is supported in a cube $Q$ centered at the origin.
We let $Q^{\ast}$ be the cube with side length $2\sqrt{n}l(Q)$ having the same center as $Q$, where the $l(Q)$ is the side length of $Q$.
By the definition of Wiener amalgam space, we have
\begin{equation}
  \begin{split}
      \|a\|_{W_{p,q}^{n(1-1/p-1/q)}}
      =&
      \big\|\|V_{\phi}f(x,\xi)\|_{L_{\xi, q}^{n(1-1/p-1/q)}}\big\|_{L_{x,p}}
      \\
      \lesssim&
      \big\|\|V_{\phi}f(x,\xi)\|_{L_{\xi, q}^{n(1-1/p-1/q)}}\chi_{Q^{\ast}}\big\|_{L_{x,p}}
      \\
      \qquad+& \big\|\|V_{\phi}f(x,\xi)\|_{L_{\xi, q}^{n(1-1/p-1/q)}}\chi_{(Q^{\ast})^c}\big\|_{L_{x,p}}
      \\
      :=& E+F.
  \end{split}
\end{equation}
By the definition of $h_p$-atom, we deduce
\begin{equation}
  \begin{split}
    |V_{\phi}f(x,\xi)|
    =&
    \left|\int_{\mathbb{R}^n}a(y)\bar{\phi}(y-x)e^{-2\pi iy\cdot \xi}dy\right|
    \\
    \lesssim &
    \|a\|_{L_{\infty}}\|\phi\|_{L_1}\lesssim |Q|^{-1/p}.
  \end{split}
\end{equation}
Thus,
\begin{equation}
  \begin{split}
    \|V_{\phi}f(x,\xi)\|_{L_{\xi, q}^{n(1-1/p-1/q)}}
    \lesssim &
    |Q|^{-1/p}\|\langle \cdot\rangle^{n(1-1/p-1/q)}\|_{L_q}
    \\
    \lesssim &
    |Q|^{-1/p}.
  \end{split}
\end{equation}
So, we have the estimate of term E:
\begin{equation}
  \begin{split}
    E=& \big\|\|V_{\phi}f(x,\xi)\|_{L_{\xi, q}^{n(1-1/p-1/q)}}\chi_{Q^{\ast}}\big\|_{L_{x,p}}
    \\
    \lesssim&
     |Q^{\ast}|^{1/p}\sup_{x\in \chi_{Q^{\ast}}}\|V_{\phi}f(x,\xi)\|_{L_{\xi, q}^{n(1-1/p-1/q)}}\lesssim 1.
  \end{split}
\end{equation}

Next, we turn to the estimate of term F. We first estimate $|V_{\phi}f(x,\xi)|$. Recalling $a$ is supported in $Q$, we have
\begin{equation}
  \begin{split}
    |V_{\phi}f(x,\xi)|
    =&
    \left|\int_{Q}a(y)\bar{\phi}(y-x)e^{-2\pi iy\cdot \xi}dy\right|
    \\
    \lesssim &
    \sup_{y\in Q}|\phi(y-x)|\int_{Q}|a(y)|dy
    \\
    \lesssim &
    \sup_{y\in Q}\langle x-y\rangle^{-\mathscr{L}}\|a\|_{L_1}
    \\
    \lesssim &
    \langle x\rangle^{-\mathscr{L}}|Q|^{1-1/p}
  \end{split}
\end{equation}
for $|\beta|=N+1$, $x\in (Q^{\ast})^c$.
Specially, for small $h_p$-atom, we can get an additional estimate of $|V_{\phi}f(x,\xi)|$ based on the cancelation of order $N=[n(1/p-1)]$.
Writing $\theta_{\xi}(x)=\bar{\phi}(x)e^{-2\pi ix\cdot \xi}$, we use the Taylor formula to deduce
\begin{equation}
  \begin{split}
    |V_{\phi}f(x,\xi)|
    =&
    \left|\int_{Q}a(y)\bar{\phi}(y-x)e^{-2\pi i(y-x)\cdot \xi}dy\right|
    \\
    = &
    \left|\int_{Q}a(y)\theta_{\xi}(y-x)dy\right|
    \\
    = &
    \int_{Q}\left(\theta_{\xi}(y-x)-\sum_{|\beta|\leqslant N}\frac{\partial^{\beta}(\theta_{\xi})(-x)}{|\beta|!}(y)^{\beta}\right)a(y)dy
    \\
    = &
    \int_{Q}(N+1)\left(\sum_{|\beta|=N+1}\frac{(y)^{\beta}}{|\beta|!}\int_0^1(1-t)^N(\partial^{\beta}\theta_{\xi})(-x+ty)dt\right)a(y)dy.
  \end{split}
\end{equation}
Recalling $\theta_{\xi}(x)=\bar{\phi}(x)e^{-2\pi ix\cdot \xi}$, $if |\beta|=N+1$
we deduce
\begin{equation}
  \begin{split}
    |(\partial^{\beta}\theta_{\xi})(-x+ty)|
    \lesssim &
    \langle \xi\rangle^{N+1}\langle x-ty\rangle^{-\mathscr{L}}\lesssim \langle \xi\rangle^{N+1}\langle x\rangle^{-\mathscr{L}}
  \end{split}
\end{equation}
for $x\in (Q^{\ast})^c$, $y\in Q$ and $t\in [0,1]$.
Thus, we obtain
\begin{equation}
  \begin{split}
    |V_{\phi}f(x,\xi)|
    \lesssim &
    \langle \xi\rangle^{N+1}\langle x\rangle^{-\mathscr{L}}\int_{Q}|y|^{N+1}|a(y)|dy
    \\
    \lesssim &
    \langle \xi\rangle^{N+1}\langle x\rangle^{-\mathscr{L}}|l(Q)|^{N+1}\cdot |Q|^{1-1/p}
    \\
    =&
    (l(Q)\langle \xi\rangle)^{N+1}|Q|^{1-1/p}\langle x\rangle^{-\mathscr{L}}
  \end{split}
\end{equation}
for all $x\in (Q^{\ast})^c$.
 Now, we are ready to give the estimate of term F. For big $h_p$-atom $a$, we have
\begin{equation}
  \begin{split}
    F=&
    \big\|\|V_{\phi}f(x,\xi)\|_{L_{\xi, q}^{n(1-1/p-1/q)}}\chi_{(Q^{\ast})^c}\big\|_{L_{x,p}}
    \\
    \lesssim &
    \big\|\|\langle x\rangle^{-\mathscr{L}}|Q|^{1-1/p}\|_{L_{\xi, q}^{n(1-1/p-1/q)}}\chi_{(Q^{\ast})^c}\big\|_{L_{x,p}}
    \\
    \lesssim &
    \big\|\langle x\rangle^{-\mathscr{L}}|Q|^{1-1/p}\chi_{(Q^{\ast})^c}\big\|_{L_{x,p}}\lesssim |Q|^{1-1/p}\lesssim 1,
  \end{split}
\end{equation}
where the last inequality we use the fact $|Q|\geqslant 1$.
For small $h_p$-atom $a$ supported in $Q$, we have
\begin{equation}
  \begin{split}
    \|V_{\phi}f(x,\xi)\|_{L_{\xi, q}^{n(1-1/p-1/q)}}
    =&
    \|V_{\phi}f(x,\xi)\chi_{l(Q)\langle \xi\rangle<1}(\xi)\|_{L_{\xi, q}^{n(1-1/p-1/q)}}
    \\
    +&
    \|V_{\phi}f(x,\xi)\chi_{l(Q)\langle \xi\rangle\geqslant 1}(\xi)\|_{L_{\xi, q}^{n(1-1/p-1/q)}}
    \\
    \lesssim &
    \|(l(Q)\langle \xi\rangle)^{N+1}|Q|^{1-1/p}\langle x\rangle^{-\mathscr{L}}\chi_{l(Q)\langle \xi\rangle<1}(\xi)\|_{L_{\xi, q}^{n(1-1/p-1/q)}}
    \\
    +&
    \||Q|^{1-1/p}\langle x\rangle^{-\mathscr{L}}\chi_{l(Q)\langle \xi\rangle\geqslant 1}(\xi)\|_{L_{\xi, q}^{n(1-1/p-1/q)}}
    \\
    \lesssim &
    \langle x\rangle^{-\mathscr{L}}
  \end{split}
\end{equation}
for $x\in (Q^{\ast})^c$. We obtain
\begin{equation}
  \begin{split}
    F=
    \big\|\|V_{\phi}f(x,\xi)\|_{L_{\xi, q}^{n(1-1/p-1/q)}}\chi_{(Q^{\ast})^c}\big\|_{L_{x,p}}
    \lesssim
    \|\langle x\rangle^{-\mathscr{L}}\|_{L_p}\lesssim 1.
  \end{split}
\end{equation}
Combining with the estimates of terms E and F, we have the desired conclusion that
\begin{equation}
  \|a\|_{W_{p,q}^{n(1-1/p-1/q)}}
  \lesssim
  E+F
  \lesssim 1
\end{equation}
for any $h_p$-atom $a$.
\end{proof}

\textbf{Sufficiency for $W^s_{p,q}\subset h_p$}

\medskip

\textbf{(i). For $1/q\geqslant 1/p\wedge 1/2$:} If $0<1/p=1/q\leqslant 1/2$, we invoke previous result in \cite{sobolev and modulation} to deduce
\begin{equation}\label{for proof, theorem Lp, 1}
W_{p,p}^{n(1-2/p)}=M_{p,p}^{n(1-2/p)}\subset L_p.
\end{equation}

On the other hand, we use Hausdorff-Young inequality $\|f\|_{L_p}\lesssim \|\hat{f}\|_{L_{p'}}$
and Proposition \ref{proposition, mild characterization} to obtain the embedding relation
\begin{equation}\label{for proof, theorem Lp, 2}
  W_{p,p'}\subset h_p
\end{equation}
for $0<1/p\leqslant 1/2$.
Then, by interpolation, (\ref{for proof, theorem Lp, 1}) and (\ref{for proof, theorem Lp, 2})
gives $W_{p,q}^{\alpha(p,q)}=W_{p,q}^{n(1-1/p-1/q)}\subset L_p$ for the zone $0<1/p\leqslant 1/2, 1/p\leqslant 1/q \leqslant 1-1/p$.

Finally, for any $1/q\geqslant (1-1/p)\vee 1/2$, denote $1/\tilde{q}=(1-1/p)\vee 1/2$, the embedding (\ref{for proof, theorem Lp, 2}) and Proposition \ref{proposition, Littlewood-paley type inequality}
then yield
\begin{equation}
  W_{p,q}^{\alpha(p,q)}=W_{p,q}\subset W_{p,\tilde{q}}\subset h_p.
\end{equation}

\medskip

\textbf{(ii). For $1/q<1/p\wedge 1/2$:} In this case, we recall $W_{p,q}^{\alpha(p,q)}\subset B_{p,q}$. Thus, the desired conclusion follows by
\begin{equation}
  W_{p,q}^{\alpha(p,q)+\epsilon}\subset B_{p,q}^{\epsilon}\subset h_p
\end{equation}
for any $\epsilon>0$.

\bigskip

\textbf{Sufficiency for $h_p\subset W_{p,q}^s$}

\medskip

\textbf{(i). For $1/q\leqslant 1/p\vee 1/2$:} Taking a fixed $f\in h_p (p\leqslant 1)$, we can find a collection of $h_p$-atoms $\{a_j\}_{j=1}^{\infty}$ and a sequence of complex numbers $\{\lambda_j\}_{j=1}^{\infty}$
which is depend on $f$,
such that $f=\sum_{j=1}^{\infty}\lambda_ja_j$ and
\begin{equation}
  \left(\sum_{j=1}^{\infty}|\lambda|^p\right)^{1/p}\leqslant C\|f\|_{h_p},
\end{equation}
where the constant $C$ is independent of $f$.
Using Proposition \ref{proposition, atom estimate}, we deduce
\begin{equation}
  \begin{split}
    \|f\|_{W_{p,q}^{n(1-1/p-1/q)}}
    = &
    \|\sum_{j=1}^{\infty}\lambda_ja_j\|_{W_{p,q}^{n(1-1/p-1/q)}}
    \\
    \lesssim &
    \left(\sum_{j=1}^{\infty}\|\lambda_ja_j\|^p_{W_{p,q}^{n(1-1/p-1/q)}}\right)^{1/p}
    \\
    \lesssim &
    \left(\sum_{j=1}^{\infty}|\lambda_j|^p\right)^{1/p}\lesssim \|f\|_{h_p}
  \end{split}
\end{equation}
where $p,q$ satisfy the conditions in (\ref{for proof, theorem Lp, 4}).

Obviously, we also have $h_2\subset W_{2,2}$.
Thus, by interpolation, we deduce that
\begin{equation}\label{for proof, theorem Lp, 7}
h_p \subset W_{p,q}^{n(1-1/p-1/q)}=W_{p,q}^{\beta(p,q)}\ \   \text{for}\ \  1/p \geqslant (1-1/q)\vee 1/q.
\end{equation}
Particularly, we have
\begin{equation}\label{for proof, theorem Lp, 5}
  h_p \subset W_{p,p'}
\end{equation}
for $1\leqslant p\leqslant 2$.
On the other hand, by the duality of Proposition \ref{proposition, Littlewood-paley type inequality}, we obtain
\begin{equation}\label{for proof, theorem Lp, 6}
  h_p \subset W_{p,2}
\end{equation}
for $p\in [2,\infty)$.

Finally, for any $0<1/p\leqslant 1-1/q, 1/q\leqslant 1/2$, we can find a $\tilde{q}$ such that $1/\tilde{q}=1/2\wedge (1-1/p)$.
Combining with (\ref{for proof, theorem Lp, 5}) and (\ref{for proof, theorem Lp, 6}), we deduce
\begin{equation}\label{for proof, theorem Lp, 8}
  h_p\subset W_{p,\tilde{q}}\subset W_{p,q}=W_{p,q}^{\beta(p,q)}
\end{equation}
for $0<1/p\leqslant 1-1/q, 1/q\leqslant 1/2$.
Combining with (\ref{for proof, theorem Lp, 7}) and (\ref{for proof, theorem Lp, 8}), we obtain the desired conclusion.

\medskip

\textbf{(ii). For $1/q> 1/p\vee 1/2$:} In this case, we recall $B_{p,q} \subset W_{p,q}^{\beta(p,q)}$. Thus, the desired conclusion follows by
\begin{equation}
  h_p\subset B_{p,q}^{-\epsilon} \subset W_{p,q}^{\beta(p,q)-\epsilon}
\end{equation}
for any $\epsilon>0$.

\subsection{Necessity}

\begin{proposition}\label{proposition, forbidden for wiener}
Let $0<p<\infty$, $0<q\leqslant \infty$.
Then we have
\begin{enumerate}
  \item
  $W_{p,q}^s\subset h_p \Longrightarrow l_q^{s+n/q,1} \subset l_p^{n(1-1/p),1}$ for $2\leqslant p<\infty$,
  \item
  $h_p\subset W_{p,q}^s \Longrightarrow l_p^{n(1-1/p),1} \subset l_q^{s+n/q,1}$ for $0<p\leqslant 2$.
\end{enumerate}
Moreover, the statement $(1)$ is true if $p=\infty$ and $h_p$ is replaced by $L_{\infty}$,
the statement $(2)$ is true if $p=1$ and $h_p$ is replaced by $L_{1}$.
\end{proposition}
\begin{proof}
We only give the proof for statement (1), since the other cases can be verified by a similar method.
Take the window function $\phi$ to be a nonzero smooth function with Fourier support in $B(0,10^{-10})$.
Choose a smooth function $h$ with Fourier support on $3/4\leqslant |\xi|\leqslant 4/3$,
satisfying $\hat{h}(\xi)=1$ on $7/8\leqslant |\xi|\leqslant 8/7$.
By the assumption of $h$, we have
$\Delta_jh_j=h_j$ for all $j\in \mathbb{N}$, where $\widehat{h_j}(\xi)=\hat{h}(\xi/2^j)$.

Denote $F=\sum_{j=0}^{\infty}a_jh_j$, where $\vec{a}=\{a_j\}_{j=0}^{\infty}$ is a truncated (only finite nonzero items) sequence of nonnegative real number.
Observing that
\begin{equation}
  h_p\sim F_{p,2}\subset F_{p,p}= B_{p,p},
\end{equation}
we have
\begin{equation}
\begin{split}
\|F\|^p_{h_p}
\gtrsim
\|F\|^p_{B_{p,p}}
\sim  &
\sum_{k\in \mathbb{N}}\|\Delta_jF\|_{L_p}^p
\\
= &
\sum_{k\in \mathbb{N}}a_j^p\|h_j\|_{L_p}^p
\sim
\sum_{k\in \mathbb{N}}a_j^p2^{jn(p-1)}
\sim
\|\{a_j\}\|^p_{l_p^{n(1-1/p),1}}.
\end{split}
\end{equation}

By the definition of $h$, we have at most one $j\in \mathbb{Z}^{+}$ such that $V_{\phi}h_j(x,\xi) \neq 0$ for every fixed $\xi$.
We also have $\langle\xi\rangle \sim 2^j$ when $V_{\phi}h_j(x,\xi) \neq 0$.
Moreover, for a fixed $\xi$ such that  $V_{\phi}h_j(x,\xi) \neq 0$ for some $j$, we have the following estimate:
\begin{equation}
  |V_{\phi}F(x,\xi)|
  =
  |V_{\phi}(a_jh_j)(x,\xi)|
  \lesssim
  a_j\langle x\rangle^{-\mathscr{L}},
\end{equation}
by a same argument as in subsection 4.1.
Using the fact $|\{\xi: V_{\phi}h_j(x,\xi) \neq 0\}|\sim 2^{jn}$, we deduce that
\begin{equation}
  \begin{split}
    \|V_{\phi}F(x,\xi)\|_{L_q^s}
    \lesssim &
    \langle x\rangle^{-\mathscr{L}}\left\|\{a_j|\{\xi: V_{\phi}h_j(x,\xi)\neq 0\}|^{1/q}\}\right\|_{l_q^{s,1}}
    \\
    \lesssim &
    \langle x\rangle^{-\mathscr{L}}\left\|\{a_j2^{jn/q}\}\right\|_{l_q^{s,1}}
    =\langle x\rangle^{-\mathscr{L}}\left\|\{a_j\}\right\|_{l_q^{s+n/q,1}}.
  \end{split}
\end{equation}
Now, we give the estimate for $W_{p,q}^s$ norm:
\begin{equation}
\begin{split}
\|F\|_{W_{p,q}^s}
= &
\left(\int_{\mathbb{R}^n}\|V_{\phi}F(x,\xi)\|^p_{L_{\xi,q}^s}\right)^{1/p}
\\
\lesssim &
\left\|\{a_j\}\right\|_{l_q^{s+n/q,1}} \left(\int_{\mathbb{R}^n}\langle x\rangle^{-\mathscr{L}}\right)^{1/p}
\\
\lesssim &
\left\|\{a_j\}\right\|_{l_q^{s+n/q,1}}.
\end{split}
\end{equation}
We have
\begin{equation}
\|\{a_j\}\|_{l_p^{n(1-1/p),1}}\lesssim \|F\|_{h_p}\lesssim \|F\|_{W_{p,q}^s}\lesssim \|\{a_j\}\|_{l_q^{s+n/q,1}}.
\end{equation}
By the arbitrary of $\vec{a}$, we obtain the embedding relation $l_q^{s+n/q,1} \subset l_p^{n(1-1/p),1}$.
\end{proof}

\begin{proposition}\label{proposition, forbidden for wiener by probability}
  Let $0<p<\infty$, $0<q\leqslant  \infty$.
  Then we have
\begin{enumerate}
  \item
  $W_{p,q}^s\subset h_p \Longrightarrow l_q^{s,0} \subset l_2^{0,0}$,
  \item
  $h_p\subset W_{p,q}^s \Longrightarrow l_2^{0,0} \subset l_q^{s,0}$ for $1<p< \infty$.
\end{enumerate}
Moreover, the statement $(1)$ is true if $p=1$ and $h_p$ is replaced by $L_{1}$.
\end{proposition}
\begin{proof}
  We only give the proof for statement $(1)$, since the other cases can be verified similarly.
  Choose the window function $\phi$ to be a nonzero smooth function satisfying that $\text{supp}\widehat{\phi} \subset B(0,10^{-10})$.
  Take $g$ to be a nonzero smooth function with $\text{supp}\widehat{g} \subset B(0,10^{-10})$. Denote $\widehat{g_k}(\xi)=\widehat{g}(\xi-k)$.
  Let $\vec{\omega}=\{\omega_k\}_{k\in \mathbb{Z}^n}$ be a sequence of independent random variables (for instance, one can choose the Rademacher functions).
  For a truncated sequence of positive number $\vec{a}=\{a_k\}_{k\in \mathbb{Z}^n}$, we define
  \begin{equation}
    G^{\vec{\omega}}=\sum_{k\in \mathbb{Z}^n}\omega_k a_k g_k.
  \end{equation}
  By the definition of $\phi$ and $g_k$, there are at almost one $k\in \mathbb{Z}^n$
  such that $V_{\phi}(g_k)(x,\xi)\neq 0$ for every fixed $\xi$.
  We also have $\langle\xi \rangle \sim \langle k\rangle$ when $V_{\phi}(g_k)(x,\xi)\neq 0$.
  Moreover, by the definition of short time Fourier transform, we deduce that
  \begin{equation}
    |V_{\phi}(g_k)(x,\xi)|
    \lesssim (|g_k|\ast |\phi|)(x)=(|g|\ast |\phi|)(x)\lesssim \langle x\rangle^{-\mathscr{L}},
  \end{equation}
  where we use the rapid decay of $g$ and $\phi$.
  Thus,
  \begin{equation}
  \begin{split}
    \|V_{\phi}G^{\vec{\omega}}(x,\xi)\|_{L_{\xi,q}^s}
    \lesssim &
    \langle x\rangle^{-\mathscr{L}}\left\|\{a_k|\{\xi: V_{\phi}g_k(x,\xi)\neq 0\}|^{1/q}\}\right\|_{l_q^{s,0}}
    \\
    \lesssim &
    \langle x\rangle^{-\mathscr{L}}\left\|\{a_k\}\right\|_{l_q^{s,0}},
  \end{split}
\end{equation}
where we use the fact that $|\{\xi: V_{\phi}g_k(x,\xi) \neq 0\}|\lesssim 1$.  By the definition of Wiener amalgam space, we have
\begin{equation}\label{for proof, theorem Besov, 11}
\begin{split}
  \|G^{\vec{\omega}}\|_{W_{p,q}^s}
  =&
  \left\|\|V_{\phi}G^{\vec{\omega}}(x,\xi)\|_{L_{\xi,q}^s}\right\|_{L_{x,p}}
  \\
  \lesssim &
  \|\langle x\rangle^{-\mathscr{L}}\|_{L_p}\cdot \left\|\{a_k\}\right\|_{l_q^{s,0}}
  \lesssim
  \left\|\{a_k\}\right\|_{l_q^{s,0}}.
  \end{split}
\end{equation}
On the other hand, observing that $G^{\vec{\omega}}$ is a Schwartz function, by the definition of local Hardy space, we obtian
$\|G^{\vec{\omega}}\|_{L_p}\lesssim \|G^{\vec{\omega}}\|_{h_p}$. Thus, $W_{p,q}^s\subset h_p$ implies
\begin{equation}\label{for proof, theorem Besov, 13}
  \|G^{\vec{\omega}}\|_{L_p}\lesssim \|G^{\vec{\omega}}\|_{W_{p,q}^s}.
\end{equation}
Using Khinchin's inequality, we deduce that
\begin{equation}\label{for proof, theorem Besov, 12}
  \begin{split}
    \mathbb{E}(\|G^{\vec{\omega}}\|^p_{L_p})
    = &
    \int_{\mathbb{R}^n}\mathbb{E}(|G^{\vec{\omega}}|^p)dx
    \\
    \sim &
    \int_{\mathbb{R}^n}\left(\sum_{k\in \mathbb{Z}^n}|a_kg_k|^2\right)^{\frac{p}{2}} dx
        \\
    \sim &
    \int_{\mathbb{R}^n}\left(\sum_{k\in \mathbb{Z}^n}|a_kg|^2\right)^{\frac{p}{2}} dx
    \sim
    \|\vec{a}\|^p_{l_2^{0,0}}\cdot \|g\|^p_{L^p}
    \sim
    \|\vec{a}\|^p_{l_2^{0,0}}.
  \end{split}
\end{equation}
Putting (\ref{for proof, theorem Besov, 11}) and (\ref{for proof, theorem Besov, 12}) into (\ref{for proof, theorem Besov, 13}), we obtain that
\begin{equation}
  \|\vec{a}\|^p_{l_2^{0,0}}
  \sim
  \mathbb{E}(\|G^{\vec{\omega}}\|^p_{L_p})
  \lesssim
  \mathbb{E}(\|G^{\vec{\omega}}\|^p_{W_{p,q}^s})
  \lesssim
  \left\|\{a_k\}\right\|^p_{l_q^{s,0}},
\end{equation}
which implies the desired conclusion.
\end{proof}

\bigskip

\textbf{Necessity for $W^s_{p,q}\subset h_p$:}

If $W^{s}_{p,q}\subset h_p$ holds, we obtain
\begin{equation}
  W^{s+\epsilon}_{p,q}\subset F_{p,2}^{\epsilon}\subset B_{p,q},
\end{equation}
for any $\epsilon>0$.
Using Theorem \ref{theorem, embedding between Wiener and Besov}, we deduce $s+\epsilon\geqslant \alpha(p,q)$.
Thus $s\geqslant \alpha(p,q)$ follows by letting $\epsilon\rightarrow 0$.

In addition, we use Proposition \ref{proposition, forbidden for wiener} to deduce $l_q^{s+n/q,1} \subset l_p^{n(1-1/p),1}$ for $2\leqslant p<\infty$.
If $1/q<1/p\leqslant 1/2$, we use Lemma \ref{lemma, Sharpness of embedding, for dyadic decomposition} to deduce that $s>n(1-1/p-1/q)=\alpha(p,q)$.
We also use Proposition \ref{proposition, forbidden for wiener by probability} to deduce that
$l_q^{s,0} \subset l_2^{0,0}$ for $0<p<\infty$, which implies $s>n(1/2-1/q)=\alpha(p,q)$ for $1/q<1/2$.

Combining with the above estimates. We obtain that $s\geqslant \alpha(p,q)$ for $1/q \geqslant 1/p\wedge 1/2$, and $s> \alpha(p,q)$ for $1/q < 1/p\wedge 1/2$.
\medskip

\textbf{Necessity for $h_p\subset W^s_{p,q}$:}

If $h_p\subset W^{s}_{p,q}$ holds, we have
\begin{equation}
  B_{p,q}\subset F_{p,2}^{-\epsilon} \subset W_{p,q}^{s-\epsilon},
\end{equation}
for any $\epsilon>0$.
Using Theorem \ref{theorem, embedding between Wiener and Besov}, we deduce $s-\epsilon\leqslant \beta(p,q)$.
Thus $s\leqslant \beta(p,q)$ follows by letting $\epsilon\rightarrow 0$.

In addition, we use Proposition \ref{proposition, forbidden for wiener} to deduce $l_p^{n(1-1/p),1}\subset l_q^{s+n/q,1} $ for $0< p\leqslant 2$.
If $1/2\leqslant 1/p< 1/q$, we use Lemma \ref{lemma, Sharpness of embedding, for dyadic decomposition} to deduce that $s<n(1-1/p-1/q)=\beta(p,q)$.
We also use Proposition \ref{proposition, forbidden for wiener by probability} to deduce that
$l_2^{0,0} \subset l_q^{s,0}$ for $p>2$, which implies $s<n(1/2-1/q)=\alpha(p,q)$ for $1/q>1/2$.

Combining with the above estimates. We obtain that $s\leqslant \beta(p,q)$ for $1/q \leqslant 1/p\vee  1/2$, and $s< \beta(p,q)$ for $1/q > 1/p\vee 1/2$.

\section{Inclusion relations between $W_{p,q}^s$ and $L_p$ for $1\leqslant p\leqslant \infty$}
Recalling Theorem \ref{theorem, embedding between Wiener and hp} and the fact that
$h_p\sim L_p$ for $p\in (1,\infty)$, to make clear the inclusion relations between $W_{p,q}^s$ and $L_p$ for $1\leqslant p\leqslant \infty$,
we only need to deal with the case for $p=1$ or $p=\infty$, which is just the content of Theorem \ref{theorem, embedding between Wiener and L1 or L infty}.
Now, we give the proof for Theorem \ref{theorem, embedding between Wiener and L1 or L infty}.

\textbf{Proof for Statement $(1)$:}
The sufficiency is a direct conclusion of Theorem \ref{theorem, embedding between Wiener and hp}.
In fact, using Theorem \ref{theorem, embedding between Wiener and hp} and the embedding relations $h_1\subset L_1$, we obtain
\begin{equation}
  W_{1,q}^s\subset h_1\subset L_1
\end{equation}
for $s\geqslant \alpha(1,q)$ with strict inequality when $1/q < 1/2$.

On the other hand,
observing that
\begin{equation}
  W_{1,q}^s\subset h_1\subset B_{1,q}^{-\epsilon}
\end{equation}
for $\epsilon>0$.
Then $s\geqslant -\epsilon$
follows by Theorem \ref{theorem, embedding between Wiener and Besov}.
By the arbitrary of $\epsilon$, we obtain $s\geqslant 0=\alpha(1,q)$ for $1/q\geqslant 1/2$.
In addition, we use Proposition \ref{proposition, forbidden for wiener by probability} to deduce that $l_q^{s,0}\subset l_2^{0,0}$,
which implies $s> n(1/2-1/q)=\alpha(1,q)$ for $1/q < 1/2$.

\textbf{Proof for Statement $(2)$:}
By the fact $W_{\infty,1}\subset L_{\infty}$ in Theorem B, we have
\begin{equation}
  W_{\infty,q}^{\alpha(\infty,q)}=W_{\infty,q}^{0}\subset W_{\infty,1}\subset L_{\infty},
\end{equation}
for $1/q\geqslant 1$.
Recalling $W_{\infty,q}^{s}\subset L_{\infty}$ for $s>\alpha(p,q)$, $1/q<1$, we actually have
$W_{\infty,q}^{s}\subset L_{\infty}$ for $s\geqslant \alpha(\infty,q)$ with strict inequality when $1/q < 1$.

For the necessity part, $s> \alpha(\infty,q)$ for $1/q<1$ is known in Theorem B.
In addition,
recalling $W_{\infty,q}^{s}\subset L_{\infty}\subset B_{\infty, q}^{-\epsilon}$ for any fixed $\epsilon>0$,
we use Theorem \ref{theorem, embedding between Wiener and Besov} to deduce that $s\geqslant -\epsilon$.
Letting $\epsilon \rightarrow 0$, we have $s\geqslant 0=\alpha(\infty, q)$ for $1/q\geqslant 1$.

\textbf{Proof for Statement $(3)$:}
The sufficiency is known in Theorem B except the case $1/p>1$.
By the embedding $W_{1,\infty}^{0}\subset W_{1,q}^{\beta(1,q)-\epsilon}$ for $\epsilon>0$, we actually have
\begin{equation}
  L_1\subset W_{1,\infty}^{0}\subset W_{1,q}^{\beta(1,q)-\epsilon}
\end{equation}
for all $q\neq \infty$, where $\epsilon$ is any fixed positive number.

For the necessity, $L_1\subset W_{1,\infty}^s$ implies $s\leqslant \beta(1,\infty)$ is known in Theorem B.
On the other hand, we recall that
$L_1\subset W_{1,q}^s$ implies $s<\beta(1,q)$ for $1\leqslant q<\infty$.
For $q<1$, we verify $s<\beta(1,q)$ by a contradiction argument.
In fact, if $L_1\subset W_{1,q}^{\beta(1,q)}$ holds for $q<1$, we use complex interpolation and the fact $L_1\subset W_{1,\infty}^{\beta(1,\infty)}$ to conclude
$L_1\subset W_{1,\tilde{q}}^{\beta(1,\tilde{q})}$ for some $1<\tilde{q}<\infty$, which leads to a contradiction.

\textbf{Proof for Statement $(4)$:}
By Theorem B, we know that $L_{\infty}\subset W_{\infty,q}^{\beta(\infty,q)}$ for $1/q\leqslant 1/2$.
For $1/q> 1/2$, recalling $W_{\infty,2}^{\beta(\infty,2)}\subset  W_{\infty,q}^{\beta(\infty,q)-\epsilon}$ for $1/q>1/2$ and $\epsilon>0$,
we conclude that
\begin{equation}
  L_{\infty}\subset W_{\infty,2}^{\beta(\infty,2)}\subset  W_{\infty,q}^{\beta(\infty,q)-\epsilon},
\end{equation}
for all $1/q>1/2$, where $\epsilon$ is any fixed positive number.

For the necessity, observing that $B_{\infty,q}^{\epsilon}\subset L_{\infty}\subset W_{\infty, q}$.
By Theorem \ref{theorem, embedding between Wiener and Besov}, $B_{\infty,q}^{\epsilon}\subset W_{\infty, q}^s$ implies $s\leqslant \beta(\infty,q)+\epsilon$.
Letting $\epsilon \rightarrow 0$, we obtain $s\leqslant \beta(\infty,q)$.
Now, we deal with the case $1/q>1/2$.
Recalling $L_p\subset W_{p,2}$ for $1/p\leqslant 1/2$, if $L_{\infty}\subset W_{\infty,q}^\beta(\infty,q)$ for certain $1/q>1/2$,
an interpolation argument then yields $L_{\tilde{p}}\subset W_{\tilde{p},\tilde{q}}^{\beta(\tilde{p},\tilde{q})}$ for some $1/\tilde{p}< 1/2$, $1/\tilde{q}>1/2$,
which leads a contradiction with the fact that $L_p\subset W_{p,q}^s$ implies $s<\beta(p,q)$ for $1/p<1/2$.

\begin{remark}
  Corollary \ref{corollary, weighted Hausdorff-Young inequality} follows by Theorem \ref{theorem, embedding between Wiener and hp}, \ref{theorem, embedding between Wiener and L1 or L infty},
  and Corollary \ref{corollary, mild characterization}.
  Corollary \ref{corollary, Inequality for Fourier series} follows by Theorem \ref{theorem, embedding between Wiener and hp}, \ref{theorem, embedding between Wiener and L1 or L infty}, and Proposition \ref{proposition, fourier series}.
\end{remark}

\end{document}